\documentclass{amsart}

\usepackage[main=french, english]{babel}
\usepackage[utf8]{inputenc}
\usepackage[T1]{fontenc}
\usepackage{lmodern}

\usepackage{amsmath}
\usepackage{amsthm}
\usepackage{amssymb}
\usepackage{tikz}

\usepackage{imakeidx}
\usepackage{appendix}
\usepackage{tikz-cd}
\usepackage{verbatim}
\usepackage{enumitem}
\usepackage{multicol}
\usepackage{csquotes}

\usepackage[backref=page]{hyperref}
\usepackage{cleveref}

\hypersetup{colorlinks=true,linkcolor=blue,anchorcolor=blue,citecolor=blue}

\usepackage{adjustbox}
\usepackage{mathtools}

\newtheorem{thm}{Th\'eor\`eme}[section]
\newtheorem{prop}[thm]{Proposition}
\newtheorem{lemma}[thm]{Lemme}
\newtheorem{cor}[thm]{Corollaire}

\theoremstyle{definition}

\newtheorem{situation}[thm]{Situation}

\theoremstyle{remark}
\newtheorem{rmk}[thm]{Remarque}
\numberwithin{equation}{section}

\newcommand{\Q}{\mathbb Q}
\newcommand{\C}{\mathbb C}
\newcommand{\Z}{\mathbb Z}
\newcommand{\R}{\mathbb R}
\newcommand{\A}{\mathbb A}
\renewcommand{\P}{\mathbb P}

\newcommand{\Spec}{\operatorname{Spec}}
\newcommand{\Br}{\operatorname{Br}}
\newcommand{\Ker}{\operatorname{Ker}}

\newcommand{\mc}[1]{\mathcal{#1}}
\newcommand{\cl}{\overline}

\renewcommand{\phi}{\varphi}

\newcommand{\on}[1]{\operatorname{#1}}

\newcommand{\puiseux}{\mathbb{R}\{\!\{t\}\!\}}

\title[Vari\'et\'es r\'eelles connexes non stablement rationnelles]{Vari\'et\'es r\'eelles semi-alg\'ebriquement connexes non stablement rationnelles}

\makeatletter
\@namedef{subjclassname@2020}{\textup{2020} Mathematics Subject Classification}
\makeatother



\makeatletter

\author{Jean-Louis Colliot-Th\'el\`ene}
\author{Alena Pirutka}
\author{Federico Scavia}

\def\@setauthors{%
  \begingroup
  \def\thanks{\protect\thanks@warning}%
  \setbox\@tempboxa=\vbox{%
    \centering
    \lineskip=1.5ex
    \normalfont\scshape Jean-Louis Colliot-Th\'el\`ene, Alena Pirutka et Federico Scavia\par
  }%
  \vspace{3.5ex}
  \noindent\box\@tempboxa
  \endgroup
}

\def\@centered@footer{\reset@font\normalfont\hfil\thepage\hfil}

\def\ps@headings{%
    \let\@mkboth\@gobbletwo 
    \def\@evenhead{\normalfont\scshape \hfill Jean-Louis Colliot-Th\'el\`ene, Alena Pirutka et Federico Scavia\hfill}%
    \def\@oddhead{\normalfont\scshape \hfill Vari\'et\'es r\'eelles connexes non stablement rationnelles\hfill}%
    \def\@oddfoot{\@centered@footer}%
    \let\@evenfoot\@oddfoot
}

\def\ps@plain{%
    \let\@mkboth\@gobbletwo
    \let\@oddhead\@empty
    \let\@evenhead\@empty
    \def\@oddfoot{\@centered@footer}%
    \let\@evenfoot\@oddfoot
}

\makeatother

\date{\today} 

\newcommand{\headerauthors}{Jean-Louis Colliot-Th\'el\`ene, Alena Pirutka et Federico Scavia}
\newcommand{\headershorttitle}{Vari\'et\'es r\'eelles connexes non stablement rationnelles}

\begin{document}

\selectlanguage{english} 
\begin{abstract}
  Let $R$ be the field of real Puiseux series. It is a real closed field. We construct the first examples of smooth intersections of two quadrics in $\P_R^5$ and smooth cubic hypersurfaces in $\P_R^4$ which are not stably rational but for which the space $X(R)$ of $R$-points is semi-algebraically connected. The question of constructing such examples over the field of real numbers $\R$ remains open.

\vspace{0.66cm} 
  \selectlanguage{french} 
  \item[\hskip\labelsep\scshape\abstractname.] 
  Soit $R$ le corps des s\'eries de Puiseux r\'eelles. C'est un corps r\'eel clos. On construit les premiers exemples d'intersections lisses de deux quadriques dans $\P_R^5$ et d'hypersurfaces cubiques lisses dans $\P_R^4$ qui ne sont pas stablement rationnelles mais pour lesquelles l'espace $X(R)$ des $R$-points est semi-alg\'ebriquement connexe. La question de construire de tels exemples sur le corps des r\'eels $\R$ reste ouverte.
\end{abstract}
\selectlanguage{french}
\maketitle

\markboth{\headerauthors}{\headershorttitle}

\pagestyle{headings}

\section*{Introduction}

Soit $X$ une $\R$-vari\'et\'e projective et lisse, g\'eom\'etri\-quement connexe,
g\'eom\'etri\-quement unirationnelle.
Pour $X$ de dimension $2$, on sait depuis Comessatti \cite{comessatti1912fondamenti} que si  l'espace topologique $X(\R)$ est connexe (et en particulier non vide) alors $X$ est rationnelle sur $\R$.   Ceci s'applique aux surfaces cubiques lisses dans $\P^3_{\R} $ et aux intersections lisses
de deux quadriques dans $\P^4_{\R}$.  
En dimension quelconque, que $X(\R)$  soit connexe  
est une condition n\'ecessaire pour que $X$ soit stablement rationnelle sur $\R$.
Ce n'est pas une condition suffisante de rationalit\'e, d\'ej\`a pour $X$
intersection lisse de deux quadriques dans  $\P^5_\R$ 
\cite{hassett2021rationality}. 
Mais la question  de savoir si cette condition est suffisante pour assurer la rationalit\'e stable est ouverte pour
certaines classes simples de vari\'et\'es, par exemple les intersections lisses de deux quadriques dans $\P^5_\R$, ou les hypersurfaces cubiques lisses de dimension au moins $3$.

On peut se poser la m\^eme question sur un corps r\'eel clos $R$ arbitraire. On dispose en effet d'une notion de connexit\'e semi-alg\'ebrique qui \'etend la notion usuelle de connexit\'e sur les r\'eels; voir \cite{delfs1981semialgebraicI} et \cite[Definition 2.4.2]{bochnak1998real}. Sur le corps \[\puiseux\coloneqq \cup_{n\geq1} \R(\!(t^{1/n})\!)\]
des s\'eries de Puiseux r\'eelles (qui est un corps r\'eel clos avec $t$ infiniment petit positif) nous donnons des exemples
de vari\'et\'es projectives lisses  $X$ dont le lieu $X(\puiseux)$ des points r\'eels est
semi-alg\'ebriquement connexe mais qui ne sont pas
$CH_{0}$-triviales, et en particulier ne sont pas stablement rationnelles,
ni m\^eme r\'etractilement rationnelles, parmi les vari\'et\'es des types suivants :

\begin{enumerate}
    \item[(i)] intersection lisse de deux quadriques dans $\P^5_{\puiseux}$ (th\'eor\`eme \ref{prop-intersection-quadriques}),
    \item[(ii)] hypersurface cubique lisse dans $\P^4_{\puiseux}$ (th\'eor\`eme  \ref{prop-cubique}),
    \item[(iii)] solide lisse fibr\'e en coniques sur $ \P^2_{\puiseux}$ et g\'eom\'etriquement rationnel (th\'eo\-r\`e\-me  \ref{prop-fibres-coniques}).  
\end{enumerate} 
On ne sait pas s'il existe de tels exemples sur le corps $\R$ des r\'eels.

Ces r\'esultats sont \`a mettre en perspective avec plusieurs articles r\'ecents.
 
\begin{enumerate}
    \item Sur tout corps $k$, des travaux de Hassett et Tschinkel  \cite{hassett2021rationality} (pour $k=\R$) et de Benoist et Wittenberg \cite{benoist2019intermediate} (pour $k$ arbitraire) ont \'etabli
la non-rationalit\'e de toute intersection lisse  $X$ de deux quadriques 
dans $\P^5_{k}$, d\`es que $X$ ne contient
pas de droite d\'efinie sur $k$, ce qui  sur $k=\R$ peut se produire avec $X(\R)$  connexe.
La m\'ethode est une \'elaboration sur un corps arbitraire de celle de Clemens et
Griffiths pour \'etablir la non-rationalit\'e des hypersurfaces cubiques lisses dans $\P^4_{\C}$.
Olivier Wittenberg a r\'ecemment appliqu\'e cette m\'ethode pour \'etablir   sous des hypoth\`eses g\'en\'erales la non-rationalit\'e 
de solides $X$ fibr\'es en surfaces quadriques sur la droite projective $\P^1_{\R}$, certains d'entre eux satisfaisant que $X(\R)$ est  connexe. Cette m\'ethode est puissante, mais elle est sp\'ecifique \`a la dimension $3$, et elle ne dit rien sur la  rationalit\'e stable.

\item Sur tout corps r\'eel clos $R$, deux des auteurs \cite{colliot2024certaines} ont  donn\'e des exemples de solides $X$  fibr\'es en surfaces quadriques sur $\P^1_{R}$
 qui ne sont pas $CH_{0}$-triviaux bien  que $X(R)$
soit semi-alg\'ebriquement connexe. L'invariant ici utilis\'e est le groupe de Brauer non ramifi\'e.
Les fibrations sur $\P^1_{R}$ ont des fibres g\'eom\'etriques singuli\`eres r\'eductibles.
\item Sur tout corps r\'eel clos $R$, deux des auteurs \cite{colliot2024certaines} ont \'etudi\'e une classe de
 solides $X$ fibr\'es en surfaces quadriques sur $\P^1_{R}$, \`a fibres g\'eom\'etriques int\`egres,  qui \'echappe aux deux techniques pr\'ec\'edentes. Pour une sous-classe assez large de tels solides $X$ ils ont \'etabli
 que $X$ est $CH_{0}$-triviale.  La question si sur $R=\R$ cela vaut toujours pour toutes les vari\'et\'es de la classe est ouverte.
 \item Sur le corps $\puiseux$, Benoist et le deuxi\`eme auteur \cite{benoist2024rationality} ont donn\'e
 des exemples de solides du type de \cite{colliot2024certaines} 
 dont l'espace des $\puiseux$-points est semi-alg\'ebriquement connexe et
 qui ne sont pas $CH_{0}$-triviaux. L'invariant utilis\'e
 est un invariant en cohomologie non ramifi\'ee de degr\'e $3$  introduit dans \cite{colliot2024certaines}, reposant sur le travail \cite{colliot1993groupe}.
 La non-nullit\'e de cet invariant est d\'etect\'ee par 
 une m\'ethode de sp\'ecialisation au-dessus des  corps
 $\R(\!(t^{1/n})\!)$ pour tous les $n\geq 1$.
\end{enumerate} 

Dans le pr\'esent article, nous d\'eveloppons une nouvelle m\'ethode, qui 
donne des r\'esultats pour plusieurs autres classes de vari\'et\'es. Nous partons de versions singuli\`eres $Y/\R$  des solides $X$ des types (i) (ii) (iii) pour lesquelles
$Y(\R)$ est connexe,  le lieu singulier est form\'e d'un nombre fini de points non r\'eels,
et  le groupe de Brauer non ramifi\'e de $Y$ n'est pas r\'eduit \`a celui de $\R$. De tels exemples ont \'et\'e construits dans \cite{colliot2024certaines}. 
Ensuite, nous d\'eformons une telle vari\'et\'e $Y$ en une famille $\mc{X}$ sur $\A^1_\R=\Spec (\R[t])$ \`a fibre g\'en\'erique lisse du même type. En nous appuyant sur une version semi-alg\'ebrique du th\'eor\`eme d'Ehresmann due \`a Coste et Shiota \cite{coste1992nash}, nous montrons (\S \ref{ehres})  que l'ensemble des $\puiseux$-points de la $\puiseux$-vari\'et\'e projective lisse $X\coloneqq \mc{X}\times_{\R[t]}\puiseux$  est semi-alg\'ebriquement connexe. Une m\'ethode alternative  pour obtenir ce r\'esultat (\S \ref{ehres}) repose sur le travail de Scheiderer \cite{scheiderer1994real}.
Par ailleurs nous montrons (\S \ref{parasp}) que la m\'ethode de sp\'ecialisation, initi\'ee par Voisin  et g\'en\'eralis\'ee par deux des auteurs \cite{colliot2016hypersurfaces}, s'adapte de fa\c{c}on souple \`a ce contexte 
(sans r\'esolution explicite des singularit\'es)
et \'etablit que la $\puiseux$-vari\'et\'e $X$ n'est pas
$CH_0$-triviale et donc n'est pas stablement rationnelle. La construction des exemples est d\'etaill\'ee au \S \ref{paraex}. 

\medskip

La nouvelle m\'ethode s'applique dans un cadre plus large (\S \ref{methode-2}),
o\`u le lieu singulier de la vari\'et\'e auxiliaire $Y$ n'est pas
n\'ecessairement fini, et o\`u le groupe de Brauer non ramifi\'e est remplac\'e  par un groupe de cohomologie non ramifi\'ee sup\'erieur.

Nous obtenons ainsi (\S \ref{dimsup}) des exemples de $\puiseux$-vari\'et\'es projectives lisses $X$ non $CH_0$-triviales avec $X(\puiseux)$ semi-alg\'ebriquement connexe
parmi des vari\'et\'es de dimension sup\'erieure des types suivants :
\begin{enumerate}
    \item[(iv)] fibration en quadriques lisse de dimension relative $6$ sur $\P^1_{\puiseux}$ (th\'eor\`eme  \ref{fibres-quadriques-6}),
    \item[(v)] intersection lisse de deux quadriques dans $\P^9_{\puiseux}$ (th\'eor\`eme  \ref{intersection-quadriques-p9}).
\end{enumerate}

Ici l'analyse  de la fibre singuli\`ere utilis\'ee dans la d\'eformation est plus d\'elicate, et on a besoin d'une  r\'esolution des singularit\'es explicite. 

On comparera (i) et (v) avec la rationalit\'e, sur tout corps r\'eel clos $R$, des intersections lisses  de deux quadriques dans $\P^4_R$ (Comessatti \cite{comessatti1912fondamenti}), dans $\P^6_{R}$ et dans 
un certain nombre de cas  dans $\P^{2n}_R$ (Hassett, Koll\'ar, Tschinkel \cite{hassett2022rationality}) lorsque leur lieu r\'eel est semi-alg\'ebriquement connexe.

On conclut cet article avec des calculs  (\S \ref{calcul}) des groupes de cohomologie non ramifi\'ee en bas degr\'e pour les vari\'et\'es projectives et lisses ici consid\'er\'ees. Dans beaucoup de cas, ces groupes  sont r\'eduits \`a la cohomologie du corps de base, et ce m\^{e}me
apr\`es toute extension du corps de base,
si bien qu'ils ne permettent pas de d\'etecter la non $CH_0$-trivialit\'e.

\subsection*{Notations et rappels} 
Un corps $F$ est dit {\it formellement r\'eel} s'il satisfait l'une des propri\'et\'es \'equivalentes suivantes (cf. \cite[Theorem 1.1.8]{bochnak1998real}) : 
\begin{itemize}
    \item[(i)] le corps $F$ peut \^etre ordonn\'e;
    \item[(ii)] l'\'el\'ement $-1$ n'est pas une somme de carr\'es dans $F$;
    \item[(iii)] pour tous $x_1,\ldots x_n\in F$, si $\sum x_i^2=0$, alors $x_i=0$, $i=1,\ldots n$.
\end{itemize}
 
Un corps r\'eel $R$ est dit {\it r\'eel clos} s'il n'admet pas d'extension alg\'ebrique r\'eelle non triviale. Il est \'equivalent de dire (cf. \cite[Theorem 1.2.2]{bochnak1998real}) que le corps $R[i]\coloneqq R[x]/(x^2+1)$ est alg\'ebriquement clos.
Par exemple, les corps suivants sont des corps r\'eels clos: le corps des nombres r\'eels $\mathbb R$, le corps $\mathbb R_{alg}$ des nombres r\'eels qui sont alg\'ebriques sur $\mathbb Q$, et le corps des s\'eries de Puiseux \`{a} coefficients r\'eels $\puiseux$ (voir \cite[Example 1.2.3.]{bochnak1998real}).

\medskip

Soit $k$ un corps.
Une $k$-vari\'et\'e est un $k$-sch\'ema s\'epar\'e de type fini. Une $k$-vari\'et\'e int\`egre
est dite {\it $k$-rationnelle} si elle est $k$-birationnelle \`a un espace projectif $\P^n_{k}$.
Une $k$-vari\'et\'e int\`egre $X$ est dite {\it stablement $k$-rationnelle} s'il existe des espaces projectifs
$\P^n_{k}$ et $\P^m_{k}$ tels que $X \times_{k}\P^n_{k}$ soit $k$-birationnel \`a $\P^m_{k}$.
Une $k$-vari\'et\'e int\`egre $X$ est dite {\it r\'etractilement rationnelle}
s'il existe des ouverts de Zariski non vides $U \subset X$ et $V \subset \P^m_{k}$ ($m$ convenable),
 et des $k$-morphismes $ f \colon  U \to V$ et $g \colon  V \to U$ tels que le compos\'e $g \circ f$ soit l'identit\'e de $U$.
Une $k$-vari\'et\'e int\`egre stablement $k$-rationnelle est r\'etractilement rationnelle.

\medskip

Pour $X$ une $k$-vari\'et\'e, on note $CH_{0}(X)$ le groupe de Chow des z\'ero-cycles modulo l'\'equivalence rationnelle.
Pour $X$ une $k$-vari\'et\'e propre,  l'application qui \`a un point ferm\'e $P$
associe le degr\'e $[k(P):k]$ du corps r\'esiduel $k(P)$ s'\'etend en un homomorphisme
$ \mathrm{deg}_{k} \colon  CH_{0}(X) \to \Z$. On note $A_{0}(X)$ le noyau de cet homomorphisme.

On dit qu'une $k$-vari\'et\'e propre g\'eom\'etriquement int\`egre $X$  est   (universellement) $CH_{0}$-triviale si pour tout corps $F$ contenant $k$ 
la fl\`eche $ \mathrm{deg}_{F} \colon  CH_{0}(X_{F}) \to \Z$ est un isomorphisme
(voir \cite{auel2017universal, colliot2016hypersurfaces}). 
Sous l'hypoth\`ese que la $k$-vari\'et\'e $X$, de corps des fonctions $k(X)$, est lisse et 
poss\`ede un z\'ero-cycle de degr\'e $1$, ceci est \'equivalent \`a l'\'enonc\'e : $A_{0}(X_{k(X)})=0$ \cite[Lemma 1.3]{auel2017universal}.
 
Si une $k$-vari\'et\'e g\'eom\'etriquement int\`egre propre  et lisse  est  stablement ra\-tion\-nel\-le, ou plus g\'en\'eralement r\'etractilement rationnelle, alors elle est  $CH_{0}$-triviale.

Soit $C$ un corps alg\'ebriquement clos de caract\'eristique z\'ero.
Rappelons que toute intersection compl\`ete lisse de deux quadriques dans $\P^n_C$ avec $n \geq 4$ est rationnelle,
et qu'il en est de m\^{e}me de toute vari\'et\'e projective et lisse de dimension au moins~2 fibr\'ee en quadriques sur la droite projective $\P^1_C$. Il existe des solides fibr\'es en coniques sur le plan projectif qui ne sont pas r\'etractilement rationnels (Artin--Mumford). 
Une hypersurface cubique lisse dans $\P^4_C$ n'est pas rationnelle (Clemens--Griffiths).
Soit $n\geq 4$. Certaines hypersurfaces cubiques lisses dans $\P^n_C$, comme l'hypersurface cubique de Fermat, sont $CH_0$-triviales. Certaines hypersurfaces cubiques lisses dans $\P^{2m+1}_C$ sont rationnelles. 
La rationalit\'e stable des hypersurfaces cubiques dans $\P^4_C$ est un probl\`eme enti\`erement ouvert. 

Soient $k$ un corps et un $M$ un module galoisien discret de torsion, premi\`ere \`a la caract\'eristique de $k$.
Soient
$X$ une $k$-vari\'et\'e int\`egre et $k(X)$
son corps des fonctions.  On note $$H^{i}_{\on{nr}}(k(X)/k, M) \subset H^{i}(k(X),M)$$
le sous-groupe du groupe de cohomologie galoisienne
form\'e des \'el\'ements qui sont non ramifi\'es par rapport \`a toute valuation discr\`ete sur $k(X)$ triviale sur $k$; voir \cite[\S 4]{colliot1995birational}, \cite{rost1996chow} et \cite{merkurjev2008unramified}. 
Si $\on{car}(k)=0$ et $X/k$ est projective et lisse, on a
\[\on{Br}(X)=\on{Br}_{\on{nr}}(k(X)/k)=H^2_{\on{nr}}(k(X)/k,\Q/\Z(1)),\]
o\`u $\on{Br}(X)=H^2_{\text{\'et}}(X,\mathbb{G}_m)$ est le groupe de Brauer de $X$, o\`u $\on{Br}_{\on{nr}}(k(X)/k)$ est le groupe de Brauer non ramifi\'e de $k(X)$ sur $k$ et o\`u, pour tout $j\in \Z$, on note $\Q/\Z(j)\coloneqq \varinjlim_n \mu_n^{\otimes j}$.

On utilisera dans cet article des r\'esultats de la th\'eorie alg\'ebrique des formes quadratiques \cite{lam2005introduction, kahn2008formes}. En particulier, on utilisera les propri\'et\'es des formes de Pfister (voir \cite[Chapter X]{lam2005introduction} et  \cite[Chapitre 2]{kahn2008formes}) et des formes d'Albert (voir \cite[(4.7) p. 69, Albert's Theorem 4.8]{lam2005introduction} et \cite[D\'efinition 5.7.7, Lemme 8.1.4]{kahn2008formes}).

\section{Un crit\`ere de connexit\'e semi-alg\'ebrique}\label{ehres}

Dans cette section, nous donnons un crit\`ere de connexit\'e semi-alg\'ebrique, au sens de \cite[Definition 2.4.2]{bochnak1998real}, pour la fibre g\'en\'erique, au-dessus du corps des s\'eries de Puiseux $\puiseux$, d'un morphisme projectif de $\R$-vari\'et\'es $f\colon X\to U$, o\`u $U\subset \A^1_\R=\Spec \R[t]$ est un sous-sch\'ema ouvert contenant $0$. 

\begin{thm}\label{critere-connexe}
Soit $\A^1_\R=\Spec \R[t]$ et soit $\eta\colon \Spec(\puiseux)\to \A^1_\R$ le morphisme correspondant \`a l'inclusion $\R[t]\subset \puiseux$. Soit $U\subset \A^1_\R$ un voisinage Zariski de $0$, soit $f\colon X\to U$ un morphisme projectif fid\`element plat de $\R$-vari\'et\'es, \`a fibre g\'en\'erique lisse, soit $X_0\coloneqq X|_{t=0}$ et soit $X_\eta\coloneqq X\times_{U,\eta}\Spec(\puiseux)$. Supposons que $X_0(\R)$ est non vide et contenu dans le lieu lisse de $X_0$. Le nombre de composantes semi-alg\'ebriquement connexes de $X_0(\R)$ coïncide avec le nombre de composantes semi-alg\'ebriquement connexes de $X_\eta(\puiseux)$.
\end{thm}

On donnera deux preuves de ce r\'esultat. La premi\`ere d\'emonstration est analytique et utilise un th\'eor\`eme fondamental de Nash-trivialit\'e des fibrations d\^{u} \`a Coste et Shiota \cite{coste1992nash}. La deuxi\`eme d\'emonstration utilise la variante r\'eelle de la th\'eorie de SGA 4, d\'evelopp\'ee par Scheiderer \cite{scheiderer1994real}.

\subsection{Premi\`ere preuve du th\'eor\`eme \ref{critere-connexe}}

Nous introduisons maintenant les notations n\'ecessaires pour \'enoncer le th\'eor\`eme   de Nash-trivialit\'e des fibrations
de   Coste et Shiota \cite{coste1992nash}; nos r\'ef\'erences de base sont l'article \cite{delfs1981semialgebraicII} de Delfs et Knebusch et le livre \cite{bochnak1998real} de Bochnak, Coste et Roy.

Soit $R$ un corps r\'eel clos et soit $X$  une $R$-vari\'et\'e. D'apr\`es \cite[Example 2 p. 182]{delfs1981semialgebraicII} l'ensemble $X(R)$ admet une structure naturelle d'espace semi-alg\'ebrique sur $R$, au sens de \cite[Definition 3 p. 182]{delfs1981semialgebraicII} (il s'agit donc d'espaces annel\'es sur $R$ au sens de \cite[Definition 2 p. 182]{delfs1981semialgebraicII}). Si $X$ est quasi-projective, $X(R)$ est m\^eme un espace semi-alg\'ebrique affine sur $R$; voir \cite[Remark p. 182]{delfs1981semialgebraicII}. Pour tout morphisme de $R$-vari\'et\'es $f\colon X\to Y$, l'application induite $f(R)\colon X(R)\to Y(R)$ est semi-alg\'ebrique au sens de \cite[p. 184]{delfs1981semialgebraicII}. Si le morphisme $f$ est propre, d'apr\`es \cite[Theorem 9.6]{delfs1981semialgebraicII} l'application $f(R)$ est propre au sens de \cite[Definition p. 192]{delfs1981semialgebraicII}, et en particulier elle est ferm\'ee.

Soit $X$ une $R$-vari\'et\'e quasi-projective. Alors $X(R)$ est un espace semi-alg\'ebrique affine sur $R$. On peut donc exhiber $X(R)$ comme un sous-ensemble alg\'ebrique $X(R)\subset R^n$ au sens de \cite[Definition 2.1.1]{bochnak1998real}, pour $n\geq 1$ convenable. Supposons de plus que $X(R)$ est contenu dans le lieu lisse de $X$. Alors, d'apr\`es \cite[Proposition 3.3.11]{bochnak1998real}, $X(R)$ est donc une vari\'et\'e de Nash affine au sens de \cite[p. 351]{coste1992nash}, ou de façon \'equivalente une sous-vari\'et\'e de Nash de $R^n$ au sens de \cite[Definition 2.9.9]{bochnak1998real}. Si $Y$ est une autre $R$-vari\'et\'e quasi-projective telle que $Y(R)$ soit contenu dans le lieu lisse de $Y$ et $f\colon X\to Y$ est un $R$-morphisme tel que $X(R)$ soit contenu dans le lieu lisse de $f$, l'application semi-alg\'ebrique induite $f(R)$ est une submersion de Nash, c'est-\`a-dire une application de Nash au sens de \cite[Definition 2.9.9]{bochnak1998real} dont la diff\'erentielle est surjective en tout point de $X(R)$. En particulier, $f(R)$ est ouverte. On rappelle que, par d\'efinition, toute application de Nash $p\colon M\to N$ entre vari\'et\'es de Nash affines est en particulier une application semi-alg\'ebrique, et que, par d\'efinition, $p$ est un diff\'eomorphisme de Nash si elle est bijective et $p^{-1}$ est aussi une application de Nash.

\begin{thm}\label{coste-shiota-thom} (Coste-Shiota)
    Soit $R$ un corps r\'eel clos, soit $M$ une vari\'et\'e de Nash affine, et soit $p\colon M\to R$ une submersion propre de Nash. Alors il existe un diff\'eomorphisme de Nash $h\colon p^{-1}(0)\times R\to M$ tel que $p\circ h$ soit la projection sur $R$.
\end{thm}  

\begin{proof}
    On peut supposer $M$ non vide. Comme toute submersion est ouverte et toute application propre est ferm\'ee, l'image de $p$ est ouverte, ferm\'ee et non vide, et donc $p$ est surjective. D'apr\`es \cite[Theorem A, Theorem 2.4(iii')]{coste1992nash}, il existe alors une vari\'et\'e de Nash affine $F$ et un diff\'eomorphisme de Nash $h'\colon F\times R\to M$ tel que $p\circ h'$ soit la projection sur $R$. Par restriction \`a $0\in R$, on d\'eduit un diff\'eomorphisme de Nash $p_0\colon F\to p^{-1}(0)$. Il suffit de poser $h\coloneqq h'\circ (p_0^{-1}\times \on{id}_R)$ pour conclure. (Par convention, toutes les vari\'et\'es de Nash consid\'er\'ees dans \cite{coste1992nash} sont affines; voir \cite[p. 351]{coste1992nash}.)
\end{proof}

\begin{proof}[Premi\`ere d\'emonstration du th\'eor\`eme \ref{critere-connexe}]
Posons $R\coloneqq \puiseux$. D'apr\`es \cite[Th\'eor\`eme 2.4.5]{bochnak1987geometrie}, les composantes connexes et les composantes semi-alg\'ebriquement con\-ne\-xes de $X_0(\R)$ coïncident. Par \cite[Proposition 5.3.6]{bochnak1998real}, $X_0(\R)$ et $X_0(R)$ ont le m\^eme nombre de composantes semi-alg\'e\-brique\-ment connexes. Il suffit donc de montrer que les semi-alg\'ebriques $X_0(R)$ et $X_\eta(R)$ ont le m\^eme nombre de composantes semi-alg\'e\-brique\-ment connexes. On prouvera que $X_0(R)$ et $X_\eta(R)$ sont Nash-diff\'eo\-morphes entre eux: le th\'eor\`eme en d\'ecoule parce qu'un diff\'eomorphisme de Nash est ouvert et ferm\'e et l'image d'un espace semi-alg\'ebri\-quement connexe par une application semi-alg\'ebrique est encore semi-alg\'ebri\-que\-ment connexe.

Comme la fibre g\'en\'erique du morphisme $f$ est lisse, passant \`a un ouvert de Zariski $U\subset \A^1_{\R}$ contenant $0$ plus petit si n\'ecessaire, on peut supposer que $f$ est lisse au-dessus de $U\setminus\{0\}$. Soit $V\subset X$ le plus grand sous-sch\'ema ouvert de $X$ tel que la restriction de $f$ \`a $V$ soit lisse. Comme le morphisme $f$ est plat, l'hypoth\`ese que $X_0(\R)$ est contenu dans le lieu lisse de $X_0$ entra\^ine que $X_0(\R)$ est contenu dans $V$. On a donc $V(\R)=X(\R)$, et le th\'eor\`eme d'homomorphisme d'Artin--Lang \cite[Theorem 4.1.2]{bochnak1998real} entra\^ine alors $V(R)=X(R)$. En particulier, $X(R)$ est une vari\'et\'e de Nash affine sur $R$. 
De plus, l'application $f(R)$ est une submersion de Nash propre : c'est une  submersion  car $V(R)=X(R)$ et elle est propre car le morphisme $f_R$ est propre.

Comme l'ouvert $U\subset \A^1_{\R}$ contient $0$, il existe $\epsilon\in \R_{>0}$ tel que l'intervalle ouvert $(-\epsilon,\epsilon)\subset \R$ soit contenu dans $U(\R)$. Si $Y\subset \A^1_{\R}$ est le compl\'ementaire de $U$, l'application $Y(\R)\to Y(R)$ est bijective (c'est un cas tr\`es facile du th\'eor\`eme d'homomorphisme d'Artin--Lang). On en d\'eduit que l'\'el\'ement $t\in R=\A^1_R(R)$ appartient \`a $U(R)$ et que l'intervalle ouvert $(-\epsilon,\epsilon)\subset R$ est contenu dans $U(R)$. Le morphisme $\eta\colon \operatorname{Spec}(R) \to U$ se factorise par le morphisme $\operatorname{Spec}(R) \to U_R$ correspondant \`a l'\'el\'ement $t\in U(R)$. Comme $\epsilon\in \R_{>0}$ et $t$ est infinit\'esimal, on a $-\epsilon < t<\epsilon$. L'intervalle $(-\epsilon,\epsilon)\subset R$ \'etant Nash-diff\'eomorphe \`a $R$, le th\'eor\`eme \ref{coste-shiota-thom} entra\^ine alors que $X_0(R)$ et $X_\eta(R)$ sont Nash-diff\'eomorphes entre eux, comme voulu.
\end{proof}

\subsection{Deuxi\`eme preuve du th\'eor\`eme \ref{critere-connexe}}
Le th\'eor\`eme \ref{critere-connexe} est un cas particulier du
th\'eor\`eme suivant, qui est une l\'eg\`ere extension de \cite[Corollaire 17.20 (a)]{scheiderer1994real}. Pour tout sch\'ema $X$, on note par $X_r$ son spectre r\'eel, avec la topologie engendr\'ee par les domaines de positivit\'e \cite[(0.4.1)-(0.4.2)]{scheiderer1994real}, et on \'ecrit $X_{ret}$ pour son site \'etale r\'eel \cite[Definition 1.2.1]{scheiderer1994real}, qui est la cat\'egorie $\mathrm{Et}/X$ des $X$-sch\'emas \'etales avec les recouvrements donn\'es par les morphismes r\'eels surjectifs au sens de \cite[Definition (1.1)]{scheiderer1994real}, c'est-\`a-dire les collections de morphismes $\{f_i\colon U_i\to U\}$ dans $\mathrm{Et}/X$ telles que $U_r$ soit la r\'eunion des $f_{ir}(U_{ir})$. D'apr\`es \cite[Theorem 1.3]{scheiderer1994real}, les topo\"i associ\'es \`a $X_r$ et $X_{ret}$ sont naturellement \'equivalents entre eux. Par d\'efinition, un faisceau ab\'elien $\mc{F}$ sur $X_r$ est dit constructible si $X_r$ est la r\'eunion d'un nombre fini de $K_i\subset X_r$ constructibles tels que $\mc{F}|_{K_i}$ soit localement constant avec fibres de type fini; voir \cite[Definition (A.3)]{scheiderer1994real}.

\begin{thm}\label{loc-const}
Soit $R$ un corps r\'eel clos et soit $f\colon  X\to U$ un morphisme propre de $R$-vari\'et\'es. Supposons ce morphisme  lisse en tout $x\in X(R)$. Soit $\mc{F}$ un faisceau ab\'elien constructible localement constant sur $X_r$. Alors, pour tout $n\geq 0$, le faisceau $R^n f_{r*} \mc{F}$ est constructible et localement constant sur $U_r$.
\end{thm}

\begin{proof}
Soient $\eta$ et $\xi$ deux points de $U$ tels que $\xi$ soit une sp\'ecialisation de $\eta$ et soit $\mc{G}$ un faisceau sur $U_r$. On note par $\mc{G}_{\xi}$ et $\mc{G}_{\eta}$ les fibres de $\mc{G}$ en $\xi$ et $\eta$, respectivement. Soit $V\subset U_r$ un voisinage ouvert de $\xi$. Alors $V$ contient $\eta$ et on dispose donc d'une application naturelle $\mc{G}(V)\to \mc{G}_\eta$. Par passage \`a la limite inductive en $V$, on obtient une fl\`eche canonique
\[\on{cosp}_{\eta\leadsto\xi}\colon \mc{G}_{\xi} \to \mc{G}_{\eta}\]
qui est dite application de cosp\'ecialisation; voir \cite[(5) p. 203]{scheiderer1994real} et \cite[Chapitre VIII, 7.7]{sga4II}. Par construction, pour tout voisinage ouvert $V$ de $\xi$ et tout voisinage ouvert $W$ de $\eta$ contenu dans $V$, on a un diagramme commutatif
\begin{equation}
\begin{tikzcd}\label{cosp-square}
    \mc{G}(V) \arrow[r] \arrow[d] & \mc{G}(W) \arrow[d] \\
    \mc{G}_\xi \arrow[r,"\on{cosp}_{\eta\leadsto\xi}"] & \mc{G}_\eta.
\end{tikzcd}
\end{equation}

On revient \`a la preuve du th\'eor\`eme \ref{loc-const}. Si le morphisme propre $f$ \'etait lisse, l'\'enonc\'e suivrait de \cite[Corollaire 17.20 (a)]{scheiderer1994real}. On adapte le m\^eme argument au cas g\'en\'eral. Par \cite[Theorem 17.7]{scheiderer1994real}, pour tout $n\geq 0$ le faisceau $R^n f_{r*} \mc{F}$ est constructible. Par cons\'equent, il suffit de montrer que pour toute sp\'ecialisation $\eta \leadsto \xi$ dans $U_r$, les applications de cosp\'ecialisation
\[\on{cosp}_{\eta\leadsto\xi}\colon (R^n f_{r*} \mc{F})_{\xi} \to (R^n f_{r*} \mc{F})_{\eta}\] sont des isomorphismes. On peut remplacer $U$ par la clôture sch\'ematique de $\{\eta\}$ et donc se r\'eduire au cas o\`u $U$ est int\`egre avec point g\'en\'erique $\eta$.

Soit $B$ l'enveloppe convexe de l'anneau local $O_{U,\xi}$ dans une clôture r\'eelle $R'$ du corps des fractions de $U$; donc $B$ est un anneau de valuation r\'eel clos au sens de \cite[(1.5)]{scheiderer1994real}. L'espace topologique $\on{Spec}(B)$ est ordonn\'e lin\'eairement par la relation de sp\'ecialisation; on montre qu'il est fini. Soit $p_0\subsetneq p_1\subsetneq\dots p_m$ une chaîne d’id\'eaux premiers de $B$ de longueur $m$, soit $x_i\in p_i\setminus p_{i-1}$ pour tout $i=1,\dots, m$ et soit $C\subset B$ la sous-$R$-alg\`ebre de $B$ engendr\'ee par les $x_i$. Les id\'eaux $p_i\cap C$ forment une chaîne d’id\'eaux premiers de $C$ de longueur $m$ et donc $m\leq \on{dim}(C)$. Comme $C$ est de type fini sur $R$, par le th\'eor\`eme de la dimension $\on{dim}(C)=\on{tr}_R(\on{Frac}(C))\leq \on{tr}_R(R')$, o\`u $\on{tr}$ est le degr\'e de transcendance. On conclut que $m\leq \on{tr}_R(R')$ et donc que l'ensemble sous-jacent \`a $\on{Spec}(B)$ est fini. En particulier, le point g\'en\'erique de $\on{Spec}(B)$ est ouvert. 

Le morphisme compos\'e $\Spec(B) \to \on{Spec}(O_{U,\xi})\to U$ envoie le point g\'en\'erique de $\on{Spec}(B)$ vers $\eta$ et le point ferm\'e de $\on{Spec}(B)$ vers $\xi$. Comme le morphisme $f$ est propre et de pr\'esentation finie, par le th\'eor\`eme de changement de base propre pour la topologie r\'eelle \'etale  \cite[Theorem 16.2 (a), $t=ret$]{scheiderer1994real} on peut remplacer $U$ par $\Spec(B)$. Donc $U = \Spec(B)$, o\`u $B$ est un anneau de valuation r\'eel clos, $\eta$ est le point g\'en\'erique de $U$ et $\xi$ est le point ferm\'e de $U$. On observe que $U$ (et donc $X$) n'est plus une $R$-vari\'et\'e en g\'en\'eral.

Soit $g\colon \eta \to U$ l'immersion ouverte donn\'ee par l'inclusion du point g\'en\'erique de $U$ et soit $X_\eta$ la fibre g\'en\'erique de $f$, de sorte que l'on a un carr\'e cart\'esien
\[
\begin{tikzcd}
X_\eta \arrow[r,"h"] \arrow[d,"f_\eta"] & X \arrow[d,"f"] \\
\eta \arrow[r,"g"] & U,
\end{tikzcd}
\]
o\`u les fl\`eches horizontales sont des immersions ouvertes. Le diagramme commutatif (\ref{cosp-square}) pour $\mc{G}=R^nf_*\mc{F}$, $V=U$ et $W=\{\eta\}$ prend la forme 
\[
\begin{tikzcd}
    H^n(X_r, \mc{F}) \arrow[r,"h_r^*"] \arrow[d,"\wr"]  &  H^n(X_{\eta r}, h_r^* \mc{F}) \arrow[d,"\wr"] \\
    (R^n f_{r*} \mc{F})_{\xi} \arrow[r,"\on{cosp}_{\eta\leadsto\xi}"] &  (R^n f_{r*} \mc{F})_{\eta}.
\end{tikzcd}
\]
Ici la fl\`eche du haut coïncide avec l'application induite par l'immersion ouverte $h_r$. La fl\`eche verticale de gauche est un isomorphisme car  $U_r$ (resp. $\{\eta\}$) est le plus petit ouvert de $U_r$ contenant $\xi$ (resp. $\eta$). (On rappelle que d'apr\`es \cite[(1.5)]{scheiderer1994real}, l'espace topologique sous-jacent au sch\'ema $U$ est hom\'eomorphe \`a $U_r$.)

Il suffit donc de montrer que l'application $h_r^*\colon H^n(X_r, \mc{F})\to H^n(X_{\eta r}, h_r^* \mc{F})$ est un isomorphisme. Cette application est un homomorphisme de coin dans la suite spectrale de Leray pour le morphisme $h_r\colon X_{\eta r}\to X_r$ et le faisceau $h_r^*\mc{F}$ sur $X_{\eta r}$:
\[E_2^{pq}\coloneqq H^p(X_r,R^qh_{r*}h_r^{*}\mc{F})\Longrightarrow H^{p+q}(X_{\eta r},h_r^*\mc{F}).\]
 Il suffit alors de prouver que la fl\`eche canonique
$\mc{F} \to h_{r*}h_r^* \mc{F}$ est un isomorphisme et que $R^q h_{r*} h_r^* \mc{F} = 0$ pour tout $q\geq 1$.
Ces deux assertions sont locales sur $X_r$ par rapport \`a sa topologie r\'eelle. 
Soit $(P,\alpha)\in X_r$ un point, c'est-\`a-dire que l'on a un point sch\'ematique $P\in X$ et  un  ordre $\alpha$ sur le corps r\'esiduel de $P$. Les hypoth\`eses entraînent que le morphisme $f$ est lisse en $P$. Pour montrer que les deux assertions sont satisfaites au voisinage de $P$, on peut alors remplacer $X$ par l'ouvert maximal de lissit\'e de $f$. Donc, renonçant ainsi \`a la propret\'e de $f$, on peut supposer que le morphisme $f$ est lisse. 

Remplaçant $X$ par un recouvrement $X'\to X$ dans $X_{ret}$ (c'est-\`a-dire, un morphisme \'etale tel que le morphisme induit $X'_r\to X_r$ soit surjectif) tel que $\mc{F}|_{X'}$ soit constant, on se r\'eduit au cas o\`u $\mc{F}$ est constant de fibre un groupe ab\'elien de type fini $M$. Alors
\[ R^q h_{r*} h_r^* \mc{F} = R^q h_{r*} M = R^q h_{r*}f_{\eta r}^* M \xleftarrow{\sim} f_r^* R^q g_{r*} M, \]
le dernier isomorphisme \'etant donn\'e par le th\'eor\`eme de changement de base lisse pour la topologie $t=ret$; voir
\cite[Theorem 16.11]{scheiderer1994real}. Pour appliquer ce th\'eor\`eme, il faut rappeler que, pour tout sch\'ema $S$, tout faisceau ab\'elien sur $S_{ret}$ est admissible \cite[Definition 16.4.1 (a)]{scheiderer1994real}. Comme $u_0$ est le spectre d'un corps r\'eel clos, le foncteur $g_{r*}$ est exact: si $\mc{F}_1\to \mc{F}_2$ est un morphisme surjectif de faisceaux sur $u_0$, alors comme $u_0$ n'admet pas de recouvrement non trivial dans $(u_0)_{ret}$, le morphisme $\mc{F}_1(u_0)\to \mc{F}_2(u_0)$ des sections globales est surjectif, c'est-\`a-dire, $\mc{F}_1\to \mc{F}_2$ est surjectif en tant que morphisme de pr\'efaisceaux, et donc il en est de même pour le morphisme $g_{r*}\mc{F}_1\to g_{r*} \mc{F}_2$. Puisque $g_{r*}M$ est le faisceau constant $M$ sur $U_r$, ceci ach\`eve la d\'emonstration.
\end{proof}

\begin{proof}[Deuxi\`eme d\'emonstration du th\'eor\`eme \ref{critere-connexe}]
    Soit $m_0$ (resp. $m_\eta$) le nombre de composantes connexes semi-alg\'ebriques de $X_0$ (resp. $X_\eta$). Le th\'eor\`eme \ref{loc-const} et le th\'eor\`eme de changement de base propre \cite[Theorem 16.2 ($t=ret$)]{scheiderer1994real}, appliqu\'es au faisceau constant $\mc{F}=\Z_{X_r}$ et $n=0$, donnent un isomorphisme entre $H^0(X_{0r},\Z)\cong \Z^{m_0}$ et $H^0(X_{\eta r},\Z)\cong\Z^{m_\eta}$, ce qui entraîne $m_0=m_\eta$.
\end{proof}

\subsection{Invariance birationnelle du nombre de composantes semi-alg\'e\-bri\-ques connexes} Pour calculer le nombre de composantes connexes de certaines vari\'et\'es singuli\`eres, le lemme suivant nous sera utile.

\begin{lemma}\label{connexe-invariant}
    Soit $R$ un corps r\'eel clos. Le nombre de composantes semi-alg\'ebri\-que\-ment connexes de $X(R)$ est un invariant birationnel des $R$-vari\'et\'es projectives int\`egres $X$ telles que $X(R)$ soit contenu dans le lieu lisse de $X$.
\end{lemma}

\begin{proof}
    L'invariance birationnelle du nombre de composantes semi-alg\'e\-bri\-quement connexes des $R$-points des $R$-vari\'et\'es projectives lisses est un fait classique \cite[\S 13]{delfs1981semialgebraicII}, \cite[Theorem 3.4.12]{bochnak1998real}. Il suffit donc de montrer que, pour toute $R$-vari\'et\'e projective int\`egre $X$ telle que $X(R)$ soit contenu dans le lieu lisse de $X$, il existe une vari\'et\'e projective lisse $X'$, birationnelle \`a $X$, telle que $X(R)$ et $X'(R)$ aient le même nombre de composantes semi-alg\'ebri\-quement  connexes. D'apr\`es Hironaka, il existe une vari\'et\'e projective lisse $X'$ et un morphisme projectif birationnel $X'\to X$ qui est un isomorphisme au-dessus du lieu lisse de $X$. Comme le lieu lisse de $X$ contient $X(R)$, on d\'eduit que $X(R)$ et $X'(R)$ sont semi-alg\'ebriquement isomorphes, et donc ils ont le même nombre de composantes semi-alg\'ebriquement connexes, comme voulu.
\end{proof}

\begin{rmk}\label{fibres-connexes}
    Soient $R$ un corps r\'eel clos et $f\colon X\to Y$ un morphisme ferm\'e ou ouvert d'espaces semi-alg\'ebriques sur $R$. Supposons que pour tout $y\in Y$ la fibre $f^{-1}(y)$ est semi-alg\'ebriquement connexe. Si $Y$ est semi-alg\'ebriquement connexe, alors $X$ l'est aussi d'apr\`es \cite[Exercise 4.4.1(a)]{scheiderer2024course}. Plus g\'en\'eralement, si $Y_1,\dots,Y_r$ sont les composantes semi-alg\'ebriquement connexes de $Y$ et, pour tout $i=1,\dots,r$, on pose $X_i\coloneqq f^{-1}(Y_i)$, alors $X_1,\dots,X_r$ sont les composantes semi-alg\'ebriquement connexes de $X$.
\end{rmk}

\section{La m\'ethode de sp\'ecialisation revisit\'ee, I}
\label{parasp}

\subsection{Sur un corps quelconque}

La m\'ethode suivie est une variante de la m\'ethode de sp\'ecialisation sur un corps quelconque, sous la forme 
d\'evelopp\'ee dans \cite[\S 1]{colliot2016hypersurfaces}.

\begin{situation}\label{situation-phi}
    Soit $L/k$ une extension de corps de caract\'eristique z\'ero, soit $Y$ une $k$-vari\'et\'e propre et g\'eom\'etriquement int\`egre, soit $S\subset Y$ le lieu singulier de $Y$ (avec structure r\'eduite), soit $U\coloneqq Y\setminus S$ l'ouvert compl\'ementaire de $S$ et soit $b\in U(k)$. Soient $Z$ une $k$-vari\'et\'e propre lisse g\'eom\'etriquement connexe et $p\colon Z \to Y$ un $k$-morphisme propre birationnel tel que la restriction $p^{-1}(U) \to U$ soit un isomorphisme. (Comme le corps $k$ est de caract\'eristique z\'ero, un tel morphisme existe d'apr\`es le th\'eor\`eme d'Hironaka.) Soit $T\coloneqq p^{-1}(S) \subset Z$. Notons $E\coloneqq k(Y)=k(Z)$ et $E'\coloneqq L(Y)=L(Z)\cong E\otimes_kL$. Soit
\begin{equation}\label{application-phi}\Phi\colon  CH_{0}(Z_{E})/ N_{E'/E}(CH_{0}(Z_{E'})) \to CH_{0}(Y_{E})/N_{E'/E}(CH_{0}(Y_{E'}))
\end{equation}
l'application induite par le morphisme propre $p$.
\end{situation}

\begin{prop}\label{prop1} 

On se place dans la situation \ref{situation-phi}. Soit $G_k$ le groupe de Galois absolu de $k$ et soit $M$ un $G_k$-module discret de torsion. Supposons:

\begin{itemize} 
\item [(i)] L'application $\Phi$ de (\ref{application-phi}) est injective.
 \item [(ii)] L'application  \[\Ker[H^i(k,M) \to H^i(L,M)] \to \Ker [H^i_{\on{nr}}(k(Y)/k,M) \to H^i_{\on{nr}}(L(Y)/L,M)]\] n'est pas
 surjective.
\end{itemize} 
  Alors la diff\'erence entre le point g\'en\'erique de $Y$
  et le point $b_{k(Y)}$ d\'efinit une classe non nulle dans $A_{0}(Y_{k(Y)})
   \subset CH_{0}(Y_{k(Y)})$.
 \end{prop}
 
\begin{proof}
L'application $\Phi$ envoie la classe du point g\'en\'erique de
$Z$ (vu comme $E$-point de $Z_E$) sur celle du point g\'en\'erique de $Y$ (vu comme $E$-point de $Y_E$). Grâce \`a l'hypoth\`ese (ii), il existe $\alpha \in \Ker [H^i_{\on{nr}}(E/k,M) \to H^i_{\on{nr}}(E'/L,M)]$ non nul, mais nul au point $b'$ image inverse de $b$ par $p\colon  Z \to Y$. D'apr\`es Merkurjev \cite[Corollary 2.9]{merkurjev2008unramified} (cf. \cite[Corollary 5.2]{schreieder2021unramified}), on dispose 
de  l'accouplement
\[CH_{0}(Z_{E})/ N_{E'/E}(CH_{0}(Z_{E'})) \times  \Ker [H^i_{\on{nr}}(E/k,M) \to H^i_{\on{nr}}(E'/L,M)] \to H^i(E,M)\]
qui associe au point g\'en\'erique de $Z$ (respectivement, au point $b'_{E}\coloneqq b'\times_k E$) et \`a la classe $\alpha$
l'\'el\'ement $\alpha \in H^{i}(E,M)$ (respectivement, z\'ero, par le choix de la classe $\alpha$). Ainsi la diff\'erence entre le point g\'en\'erique de $Z$ et le point constant $b'_E$
  d\'efinit une classe non triviale dans $CH_{0}(Z_{E})/ N_{E'/E}(CH_{0}(Z_{E'}))$. D'apr\`es l'hypoth\`ese (i), on en d\'eduit que
 la diff\'erence entre le  point g\'en\'erique de $Y$ (vu comme $E$-point de $Y_E$) et le point $b_{E}\coloneqq b\times_{k}E$
d\'efinit une classe de degr\'e z\'ero qui est non nulle dans $CH_{0}(Y_{E})/ N_{E'/E}(CH_{0}(Y_{E'}))$,
et a fortiori dans $CH_{0}(Y_{E})$.
\end{proof}

Le th\'eor\`eme suivant est dans l'esprit de \cite[Th\'eor\`eme 1.14]{colliot2016hypersurfaces}.

\begin{thm}\label{1.14K} 
On se place dans la situation \ref{situation-phi}. Soit  $A$ un anneau de valuation discr\`ete de corps r\'esiduel $k$, soit $K$ son corps des fractions, soit $\pi \colon  { \mathcal{X}} \to  \Spec(A)$ un $A$-sch\'ema projectif, fid\`element plat, \`a fibres g\'eom\'etriquement int\`egres, tel que $\mc{X}\times_Ak\cong Y$, soit $X$ la fibre g\'en\'erique de $\pi$ et soit $\sigma$ une section de $\pi$ qui induit un $k$-point lisse $b$ sur $Y$. Soit $M$ un $G_k$-module discret, de torsion, et soit $i\geq 0$ un entier. Supposons:
\begin{itemize} 
\item [(i)] L'application $\Phi$ de (\ref{application-phi}) est injective.
 \item [(ii)] L'application  \[\Ker[H^i(k,M) \to H^i(L,M)] \to \Ker [H^i_{\on{nr}}(k(Y)/k,M) \to H^i_{\on{nr}}(L(Y)/L,M)]\] n'est pas
 surjective.
\end{itemize} 
 Alors le groupe $A_{0}(X\times_{K}K(X))$ est non nul, en d'autres termes,
 la $K$-vari\'et\'e $X$ n'est pas $CH_{0}$-triviale.
 \end{thm}
 
\begin{proof}
On suit la strat\'egie de la preuve de \cite[Th\'eor\`eme 1.12]{colliot2016hypersurfaces}. Soit $\eta$ le point g\'en\'erique  de $Y$ et soit $B\coloneqq O_{\mc{X},\eta}$ son anneau local, qui est de dimension $1$. Comme la fibre $Y$ est g\'eom\'etriquement int\`egre, l'id\'eal maximal de $B$ est engendr\'e par l'image $\nu$ d'une uniformisante de $A$. Comme $\pi$ est fid\`element plat, $\nu$ n'est pas diviseur de z\'ero, et donc $B$ est un anneau de valuation discr\`ete de corps des fractions le corps des fonctions $K(X)$ de $X$ et de corps r\'esiduel le corps des fonctions $k(Y)$ de $Y$. On consid\`ere le $B$-sch\'ema ${\mathcal X}\times_{A}B$. Sa fibre g\'en\'erique est $X\times_{K(X)}$ et sa fibre sp\'eciale est $Y_{k(Y)}$.

On a un homomorphisme de sp\'ecialisation
$CH_{0}(X\times_{K(X)}) \to CH_{0}(Y_{k(Y)})$
qui envoie le point g\'en\'erique de $X$ sur le point g\'en\'erique de $Y$ et le point $\sigma(K)_{K(X)}$ sur $b_{k(Y)}$; voir \cite{fulton1975rational} et \cite[\S 20.3]{fulton1998intersection}.

D'apr\`es la proposition \ref{prop1},  la classe du point g\'en\'erique de $Y$ 
dans $CH_{0}(Y_{k(Y)})$ n'est pas 
\'egale \`a celle de $b_{k(Y)} $,
et donc $A_{0}(X\times_{K(X)}) \neq 0$. 
\end{proof}

On donne  une condition suffisante pour l'injectivit\'e de $\Phi$ qui suffira pour tous les exemples de la section \ref{paraex}. Une autre condition suffisante sera donn\'ee dans le lemme \ref{lemme-phi-isomorphisme-variante} et sera utilis\'ee pour les exemples de la section \ref{dimsup}.

\begin{lemma}\label{lemme-phi-isomorphisme}
    On se place dans la situation \ref{situation-phi}. Si le morphisme $S \to \Spec(k)$ se factorise par $\Spec(L) \to Spec(k)$, alors l'application $\Phi$ est un isomorphisme.
\end{lemma}

\begin{proof}
 On a un diagramme commutatif 
\[
\begin{tikzcd}[row sep=2em, column sep=3.5em]
CH_{0}(T_{E'}) \arrow[r] \arrow[d,->>,"N_{E'/E}"] \arrow[ddd, bend left=66,"p_*"] & 
CH_{0}(Z_{E'}) \arrow[r] \arrow[d,"N_{E'/E}"] \arrow[ddd, bend left=66,"p_*"] & 
CH_{0}(p^{-1}(U)_{E'}) \arrow[r] \arrow[d,"N_{E'/E}"] \arrow[ddd, bend left=66,"\wr"', "p_*"] & 
0 \\
CH_{0}(T_{E}) \arrow[r] \arrow[d,"p_*"] & 
CH_{0}(Z_{E}) \arrow[r] \arrow[d,"p_*"] & 
CH_{0}(p^{-1}(U)_{E}) \arrow[r] \arrow[d,"\wr"', "p_*"] & 
0 \\
CH_{0}(S_{E}) \arrow[r] & 
CH_{0}(Y_{E}) \arrow[r] & 
CH_{0}(U_{E}) \arrow[r] & 
0 \\
CH_{0}(S_{E'}) \arrow[r] \arrow[u,->>,swap,"N_{E'/E}"] & 
CH_{0}(Y_{E'}) \arrow[r] \arrow[u,swap,"N_{E'/E}"] & 
CH_{0}(U_{E'}) \arrow[r] \arrow[u,swap,"N_{E'/E}"] & 
0
\end{tikzcd}
\]
o\`u les suites exactes horizontales proviennent de la suite de localisation, et o\`u on \'ecrit $N_{E'/E}$ pour les applications normes.
 
 Par hypoth\`ese, le $k$-sch\'ema $S$ admet une structure de $L$-sch\'ema. Comme le corps $k$ est de caract\'eristique z\'ero, l'inclusion diagonale $L\to L\otimes_kL$ est scind\'ee, et donc la projection de $L$-sch\'emas $S\times_kL=S\times_L(L\otimes_kL)\to S$ admet une section. Ainsi la fl\`eche norme $N_{E'/E}\colon CH_{0}(S_{E'}) \to CH_{0}(S_{E})$ est surjective. Comme $T$ admet un morphisme vers $S$, il admet aussi une structure de $L$-sch\'ema, et le même argument montre que la fl\`eche norme $N_{E'/E}\colon CH_{0}(T_{E'}) \to CH_{0}(T_{E})$ est surjective. 
 On d\'eduit alors du diagramme ci-dessus que l'application $\Phi$ est un isomorphisme.
\end{proof}

\section{Vari\'et\'es projectives et lisses non stablement rationnelles sur le corps \texorpdfstring{$R$}{R} des s\'eries de Puiseux r\'eelles dont l'espace des \texorpdfstring{$R$}{R}-points est semi-alg\'ebriquement connexe}
\label{paraex}

\begin{thm}\label{generateur}
Soit $Y$ une $\R$-vari\'et\'e projective et g\'eom\'etriquement int\`egre.  Soit $S\subset Y$ le lieu singulier de $Y$, avec structure r\'eduite. Supposons:
\begin{enumerate}
    \item[(a)] L'ensemble $Y(\R)$ des points r\'eels est contenu dans le lieu lisse de $Y$.
    \item[(b)] L'espace topologique $Y(\R)$ est connexe.
    \item[(c)] Le morphisme $S\to \on{Spec}(\R)$ se factorise par $\Spec(\C)\to\Spec(\R)$. 
    \item[(d)] L'application
\[\Br(\R) \to \Ker [\Br_{\on{nr}}(\R(Y)/\R) \to \Br_{\on{nr}}(\C(Y)/\C)]\] n'est pas
 surjective.
\end{enumerate}

Soit $\A^1_\R=\Spec( \R[t])$, soit $U\subset \A^1_\R$ un voisinage Zariski de $0$, soit $f\colon \mc{X}\to U$ un morphisme projectif de $\R$-vari\'et\'es, \`a fibre g\'en\'erique lisse, tel que $\mc{X}|_{t=0}\cong Y$. Alors, si $\eta\colon \Spec(\puiseux)\to \A^1_\R$ est le morphisme correspondant \`a l'inclusion $\R[t]\subset R$, la $\puiseux$-vari\'et\'e projective lisse $X\coloneqq \mc{X}\times_{U,\eta}\Spec(\puiseux)$ satisfait les propri\'et\'es suivantes.

\begin{enumerate}
    \item[(i)] L'espace semi-alg\'ebrique $X(\puiseux)$ est semi-alg\'ebriquement connexe.
    \item[(ii)] La $\puiseux$-vari\'et\'e $X$ n'est pas $CH_0$-triviale, et en particulier
n'est pas stablement rationnelle.
\end{enumerate} 
\end{thm}
 
\begin{proof}
(i) Ceci suit du th\'eor\`eme \ref{critere-connexe} et des propri\'et\'es (a) et (b).

(ii) Si la $\puiseux$-vari\'et\'e $X$ est $CH_0$-triviale,
alors il existe un entier $n>0$ tel que la vari\'et\'e 
$\mc{X}\times_U \R(\!(t^{1/n})\!)$ sur le corps $\R(\!(t^{1/n})\!)$
soit $CH_0$-triviale. Ceci r\'esulte de la forme ``d\'ecomposition de la diagonale''
de la $CH_0$-trivialit\'e \cite[Proposition 1.4]{colliot2016hypersurfaces}. Pour montrer (ii), il suffit donc de prouver que pour tout $n\geq 1$ la vari\'et\'e 
$\mc{X}\times_U \R(\!(t^{1/n})\!)$ n'est pas $CH_0$-triviale. 

Soit $n\geq 1$ un entier, soit $A\coloneqq \R[\![t^{1/n}]\!]$ et soit $\mc{X}'\coloneqq \mc{X}\times_U A$, o\`u le produit fibr\'e est pris par rapport au morphisme \'evident $\Spec(A)\to U$. La fibre sp\'eciale du morphisme naturel $\pi\colon \mc{X}'\to \Spec(A)$ est isomorphe \`a $Y$. D'apr\`es (a) et (b), il existe un $\R$-point lisse $y\in Y(\R)$, qui se rel\`eve en une section $\sigma$ de $\pi$ d'apr\`es le lemme de Hensel. Ceci, les propri\'et\'es (c) et (d) et le lemme \ref{lemme-phi-isomorphisme} nous garantissent que $\pi$ satisfait les hypoth\`eses (i) et (ii) (avec $i=2$ et $M=\mu_2$) du th\'eor\`eme \ref{1.14K}, ce qui entraîne que la fibre g\'en\'erique $(\mc{X}')_{\R(\!(t^{1/n})\!)}\cong\mc{X}_{\R(\!(t^{1/n})\!)}$ n'est pas $CH_0$-triviale, comme voulu. 
\end{proof}
 
Ce th\'eor\`eme va nous permettre de construire des exemples projectifs et lisses $X$ sur $\puiseux$ en partant d'exemples projectifs singuliers $Y$ sur $\R$. Les trois premiers exemples singuliers avaient \'et\'e construits dans \cite{colliot2024certaines}.

\subsection{Fibrations en surfaces quadriques sur la droite projective dont toutes les fibres g\'eom\'etriques sont int\`egres}\label{contreexemple}

Soit $\P^1_\R=\on{Proj}(\R[u_0,u_1])$, vu comme recollement de $\A^1_\R$ avec coordonn\'ee $u\coloneqq u_0/u_1$ et de $\A^1_\R$ avec coordonn\'ee $v\coloneqq u_0/u_1$ au moyen du changement de variable $u=1/v$ sur $\P^1_\R\setminus\{0,\infty\}$. Soit $Y$ le recollement de la sous-vari\'et\'e de $\P^3_{\R} \times_{\R} \A^1_{\R}$ avec coordonn\'ees $(a,b,c,d;u)$ donn\'ee par l'\'equation
\[a^2+(1+u^2) b^2 - u (c^2+d^2)=0.\]
  avec la sous-vari\'et\'e de $\P^3_{\R} \times_{\R} \A^1_{\R}$ avec coordonn\'ees  $(a',b',c',d'; v)$ donn\'ee par l'\'e\-qua\-tion
\[a'^2+(1+v^2)b'^2-v(c'^2+d'^2)=0\]
  au moyen du  changement de variable 
 \[(a',b',c',d';v)=(a/u, b, c, d; 1/u).\] On a un morphisme surjectif $\pi\colon Y \to \P^1_\R$ donn\'e par $(u_0:u_1)$.
  
  \begin{lemma}\label{refX}
  Soit $\pi \colon  Y \to \P^1_{\R}$ la fibration en surfaces quadriques d\'efinie ci-dessus. Alors $Y$ satisfait les propri\'et\'es (a)-(d) du th\'eor\`eme \ref{generateur}. 
 \end{lemma}
 \begin{proof} 
Les propri\'et\'es (a) et (c) se v\'erifient en calculant explicitement le lieu singulier de $Y$, qui est form\'e d'un nombre fini de points dont aucun n'est r\'eel. Pour montrer (b) (la connexit\'e), il suffit de noter que l'application $Y(\R) \to \P^1(\R)$ a  pour image $[0,\infty]$, qui est connexe, et que pour tout $s\in [0,\infty]$ la fibre en $s$ est connexe: pour $0<s<\infty$, les points r\'eels des fibres sont des quadriques connexes, et pour $s=0,\infty$, ce sont des droites projectives. Pour (d), il suffit de rappeler que, d'apr\`es \cite[Proposition 11.1]{colliot2024certaines}, l'image de la classe de quaternions $(-1,u) \in \Br(\R(u))=\Br(\R(\P^1))$ dans $\Br(\R(Y))$ est non ramifi\'ee, s'annule dans $\Br(Z_\C)$, et n'est pas dans l'image de $\Br(\R)$.
  \end{proof} 
 
Soit $P(u)=a_0 + a_1u + a_2 u^2 \in \R[u]$, avec $a_0a_2\neq 0$.
On consid\`ere le sous-sch\'ema ferm\'e de $\P^3_{\R[t]} \times_{\R[t]} \A^1_{\R[t]}$ avec coordonn\'ees $(a,b,c,d;u)$ donn\'e par l'\'equation
  \[a^2+(1+u^2)b^2 + (tP(u)- u)c^2 -ud^2=0.\]
  On  recolle ce sch\'ema avec le sous-sch\'ema ferm\'e de $\P^3_{\R[t]} \times_{\R[t]} \A^1_{\R[t]}$ avec coordonn\'ees $(a',b',c',d';v)$
  donn\'e par l'\'equation
  \[a'^2+(1+v^2)b'^2 +  (tQ(v)-v)c'^2 -vd'^2=0,\]
  o\`u $Q(v)\coloneqq v^2P(1/v)= a_0v^2 + a_1v + a_2$,  via le changement de variables 
  \[(a',b',c',d';v)=(a/u, b, c, d; 1/u).\]
  On obtient une fibration en quadriques $\mc{X} \to \P^1_{\R[t]}$. La fibre sp\'eciale en $t=0$ est la $\R$-vari\'et\'e $Y$. La fibre g\'en\'erique sur le corps $\R(t)$
  est une vari\'et\'e projective lisse munie d'une fibration en quadriques sur $\P^1_{\R(t)}$ dont toutes les fibres g\'eom\'etriques sont int\`egres.

  Le th\'eor\`eme \ref{generateur} et le lemme \ref{refX} donnent maintenant :
  \begin{thm}\label{prop-fibre-quadriques}
Soit $\mc{X}\to \P^1_{\R[t]}$ le morphisme d\'efini ci-dessus.
La $\puiseux$-vari\'et\'e   $X\coloneqq \mc{X}\times_{\R[t]} \puiseux$ est une vari\'et\'e projective lisse fibr\'ee en surfaces quadriques sur $\P^1_{\puiseux}$ \`a fibres g\'eom\'etriques int\`egres. Elle n'est pas $CH_0$-triviale et donc n'est pas stablement rationnelle. L'espace $X(\puiseux)$ est semi-alg\'ebriquement connexe.
  \end{thm}

Soit $R$ un corps r\'eel clos. Dans \cite{colliot2024certaines} et \cite{benoist2024rationality} on a \'etudi\'e des fibrations en surfaces quadriques $X'\to \P^1_R$ d'\'equation affine
\begin{equation}\label{u-p(u)} x^2+y^2+z^2= u p(u)\end{equation}
avec $p(u)\in R[u]$ un polynôme s\'eparable. Si $p(u)$ est strictement positif sur $R$, l'espace $X'(R)$ est semi-alg\'ebriquement connexe. 
  
  \begin{prop}
  La fibre g\'en\'erique de la fibration $X \to \P^1_{\puiseux}$ de la proposition \ref{prop-fibre-quadriques} n'est isomorphe \`a la fibre g\'en\'erique d'aucune fibration $X'\to \P^1_{\puiseux}$ de mod\`ele affine (\ref{u-p(u)}). 
   \end{prop}
   \begin{proof}
Soient plus g\'en\'eralement $k$ un corps de caract\'eristique diff\'erente de $2$ et $p\colon X \to \P^1_{k}$ une fibration en surfaces quadriques dont toutes les fibres 
g\'eom\'etriques sont int\`egres. On suppose qu'il existe au moins une fibre
g\'eom\'etrique singuli\`ere. 
 \`{A} une telle fibration on associe le rev\^{e}tement discriminant  $\Delta \to \P^1_{k}$.
La courbe $\Delta$ est projective et lisse, et munie d'une classe
$\alpha \in  \Br(\Delta)$ qui est obtenue de la fa\c{c}on suivante.
On prend une section lisse quelconque de la quadrique g\'en\'erique sur $k(\P^1)$
et on l'\'etend  au corps $k(\Delta)$. On obtient alors une quadrique sur
$k(\Delta)$ d\'efinie par une forme quadratique $\langle 1,-a,-b,ab\rangle$. La classe $\alpha$ est la classe de l'alg\`ebre de quaternions $(a,b)\in \on{Br}(k(\Delta))$. Pour plus de d\'etails, on consultera
\cite[Th\'eor\`eme 2.5]{colliot1993groupe}.
L'hypoth\`ese sur les fibres singuli\`eres
assure que l'on a $\alpha \in \Br(\Delta)$. 
Pour les familles $X \to \P^1_{k}$ consid\'er\'ees
dans \cite{colliot2024certaines} et \cite{benoist2024rationality}, la classe $\alpha$ est dans l'image
de $\Br(k) \to \Br(\Delta)$. De fait, pour une fibration d'\'equation affine
$x^2+y^2+t^2=P(u)$, avec $P(u)\in k[u]$, la classe $\alpha \in \Br (k(\Delta))$ est l'image de $(-1,-1) \in \Br(k)$.

Soit $R=\puiseux$. Dans le cas consid\'er\'e au th\'eor\`eme 
\ref{prop-fibre-quadriques}, avec $k=R$,
la courbe $\Delta$ est d\'efinie sur $R$ par l'\'equation
$$ w^2= (1+u^2)(-tuP(u)+u^2).$$
On prend  la section de la quadrique g\'en\'erique donn\'ee par $c=0$. La classe
$\alpha \in \Br(\Delta)$ est donc l'image de $(u,-(1+u^2))=(-1,u) \in \Br(R(\P^1))$ dans 
$\Br(R(\Delta))$.   En un point $A \in \Delta(\R) \subset \Delta(R)$ avec $u>0$, 
on a $\alpha(A)=0$. En un point $B \in \Delta(\R) \subset \Delta(R)$
avec $u<0$ on a $\alpha(B)\neq 0$. Ainsi $\alpha$ n'est pas dans
l'image de $\Br(R)$.
  \end{proof}

\subsection{Intersections compl\`etes lisses de deux quadriques dans \texorpdfstring{$\P^5$}{P5}}
Consid\'erons l'espace projectif $\P^5_\R$ avec coordonn\'ees homog\`enes $a,b,c,d,e,f$. La vari\'et\'e du paragraphe \ref{contreexemple} contient l'ouvert affine d\'efini par l'\'equation 
\[a^2+1+u^2 - u (c^2+d^2)=0,\]
ou encore le syst\`eme de deux \'equations affines
\begin{align*}
    a^2+1+u^2 -uv &=0,  \\
    c^2+d^2-v&=0.
\end{align*}
En changeant les notations, on voit que 
la vari\'et\'e ainsi d\'efinie est birationnelle 
\`a l'intersection de
deux quadriques  $Y\subset \P^5_{\R}$
donn\'ee par les \'equations homog\`enes
\begin{align*}
    a^2+ b^2+c^2   - cd&=0,\\
    e^2+f^2-bd&= 0.
\end{align*}
Notons que $Y \subset \P^5_\R$ ne contient pas de droite r\'eelle: en effet la section de $Y$ par $d=0$ est donn\'ee par $a^2+ b^2+c^2=e^2+f^2=0$ et donc ne contient pas de point r\'eel. (Cette remarque ne sera pas utilis\'ee dans la suite.)

\begin{lemma}\label{speciale-int-quadriques}
    Soit $Y\subset \P^5_\R$ l'intersection de deux quadriques d\'efinies ci-dessus. Alors $Y$ satisfait les propri\'et\'es (a)-(d) du th\'eor\`eme \ref{generateur}. 
\end{lemma}

\begin{proof}
    Le lieu singulier de $Y$ est form\'e des $4$ points non r\'eels donn\'es par 
$$(a,b,c,d,e,f)= (\pm i, 0, 1,0,0,0),\, (0,0,0,0,\pm i, 1).$$
Ceci entraîne (a) et (c). Comme $Y$ est birationnelle \`a la vari\'et\'e du paragraphe \ref{contreexemple}, la propri\'et\'e (b) suit des lemmes \ref{connexe-invariant} et \ref{refX}, et la propri\'et\'e (d) suit du lemme \ref{refX}.
\end{proof}

Soient $q_1, q_2$ deux formes quadratiques sur $\R$  en les 6 variables $(a,b,c,d,e,f)$. Supposons 
que le syst\`eme $q_1=q_2=0$ d\'efinisse une intersection compl\`ete lisse de deux quadriques dans $\P^5_{\R}$. On sait qu'il suffit pour cela que le d\'eterminant $\on{det}(q_1+tq_2)$  soit un polyn\^{o}me s\'eparable de degr\'e $6$.

Soit $\mc{X}\subset \P^5_{\R[t]}$ d\'efini par le syst\`eme
\begin{align*}
    a^2+ b^2+c^2   - cd +tq_1 &=0,\\
    e^2+f^2 -bd  +t q_2 &=0.
\end{align*}
Alors $\mc{X}|_{t=0}=Y$ et la fibre g\'en\'erique sur $\R(t)$ est une intersection compl\`ete lisse de deux quadriques. Le th\'eor\`eme \ref{generateur} et le lemme \ref{speciale-int-quadriques} donnent ici :

\begin{thm}\label{prop-intersection-quadriques}
La $\puiseux$-vari\'et\'e $X= \mc{X}\times_{\R[t]} \puiseux$ est une intersection compl\`ete lisse de deux quadriques dans $\P^5_{\puiseux}$ qui n'est pas $CH_0$-triviale, et donc n'est pas
stablement rationnelle. L'espace $X(\puiseux)$ est semi-alg\'ebriquement connexe.\end{thm}

 \begin{rmk}
 On ne peut envisager une extension simple de l'argument ci-dessus aux intersections de deux quadriques en dimension sup\'erieure. On sait en effet 
 que sur tout corps $k$ de caract\'eristique z\'ero, pour toute vari\'et\'e $Y$ intersection compl\`ete g\'eom\'etriquement int\`egre non conique de deux quadriques dans $\P^n_k$ avec $n\geq 6$, et tout mod\`ele projectif et lisse $Z$ de $Y$, l'application
 $\Br(k) \to \Br(Z)$ est surjective \cite[Theorem 3.8]{colliot1987intersectionsI}.
 Par ailleurs, Hassett, Koll\'ar et Tschinkel 
\cite{hassett2022rationality} ont montr\'e que toute intersection compl\`ete lisse de deux quadriques $X \subset \P^6_{\R} $
 est $\R$-rationnelle d\`es que $X(\R)$ est connexe. On v\'erifie que leur d\'emonstration vaut sur tout corps r\'eel clos $R$.
 \end{rmk}

 \subsection{Hypersurfaces cubiques lisses dans \texorpdfstring{$\P^4$}{P4} }\label{exemplecubique}

La vari\'et\'e $X$ de \S\ref{contreexemple} est $\R$-bi\-ra\-tion\-nelle
\`a 
 l'hypersurface cubique singuli\`ere $Y \subset \P^4_{\R}$, g\'eom\'etri\-quement int\`egre, d\'efinie
par l'\'equation
\begin{equation}\label{Ycub}
b(a^2+b^2+c^2) - c(d^2+e^2)=0
\end{equation}  
en coordonn\'ees homog\`enes $a,b,c,d,e$. 

\begin{lemma}\label{cubique}
    L'hypersurface cubique singuli\`ere $Y \subset \P^4_{\R}$  d\'efinie ci-dessus satisfait les hypoth\`eses (a)-(d) du th\'eor\`eme \ref{generateur}.
\end{lemma}

\begin{proof}
Le lieu singulier de $Y$ est form\'e des $4$ points non r\'eels donn\'es par 
$$(a,b,c,d,e) = (\pm i, 0,1,0,0),\, (0,0,0, \pm i, 1).$$
Ceci entraîne (a) et (c). Comme $Y$ est birationnelle \`a la vari\'et\'e du paragraphe \ref{contreexemple}, la propri\'et\'e (b) suit des lemmes \ref{connexe-invariant} et \ref{refX}, et la propri\'et\'e (d) suit du lemme \ref{refX}.
\end{proof}

Soit  $g(a,b,c,d,e)$ une forme cubique non singuli\`ere. Soit $\mc{X} \subset \P^4_{\R[t]}$  le sch\'ema d\'efini par l'\'equation
 $$b(a^2+b^2+c^2) - c(d^2+e^2) -tg(a,b,c,d,e)=0.$$
On a $\mc{X}|_{t=0}=Y$. Le th\'eor\`eme \ref{generateur} et le lemme \ref{cubique} donnent ici :

\begin{thm}\label{prop-cubique}
La $\puiseux$-vari\'et\'e $X\coloneqq \mc{X}\times_{\R[t]} \puiseux$ est une hypersurface cubique lisse  dans $\P^4_{\puiseux}$ qui n'est pas $CH_0$-triviale, et donc n'est pas
stablement rationnelle. L'espace $X(\puiseux)$ est semi-alg\'ebriquement connexe.
\end{thm}

\begin{rmk}
Pour les hypersurfaces cubiques singuli\`eres sur $\R$, des r\'esultats de non-rationalit\'e ont \'et\'e obtenus par Cheltsov, Tschinkel et Zhang \cite{cheltsov2024rationality}.
\end{rmk}

\subsection{Certaines fibrations en coniques  sur \texorpdfstring{$\P^2$}{P2}.}

Soit $R$ un corps r\'eel clos et soit $C$ sa clôture alg\'ebrique. Dans \cite{benoist2020clemens}, Benoist et Wittenberg consid\`erent les solides lisses $X$ fibr\'es en coniques sur
$\P^2_R$, donn\'es par des \'equations (voir ci-dessous des explications sur la notation)
$$s^2+v^2=f(x,y,z)$$
avec $f$ une forme homog\`ene de degr\'e $4$, positive,
d\'efinissant une courbe lisse dans $\P^2_R$. Pour tout solide $X$ de cette forme, $X(R)$ est semi-alg\'ebriquement connexe (cf. remarque \ref{fibres-connexes}) et $X_C$ est $C$-rationnel. Benoist et Wittenberg montrent que $X$ n'est pas $R$-rationnel; voir \cite[Corollary 3.6]{benoist2020clemens}. Nous allons montrer qu'il existe de tels solides lisses $X$ sur le corps $\puiseux$ qui ne sont pas $CH_0$-triviaux, donc ne sont pas stablement rationnels, et pour lesquels $X(\puiseux)$ est semi-alg\'ebriquement connexe, et qui sont g\'eom\'etriquement rationnels.

Soit $A$ un anneau, soit $S\coloneqq \P^2_A$, soit $d\geq 0$ un entier et soit $F \in H^0(S, \mc{O}(2d))$ une section avec lieu d'annulation $C\subset S$. Soit $\P \coloneqq \P_S(\mc{O}(-d) \oplus \mc{O}(-d) \oplus \mathcal{O})$ et soit $p\colon \P\to S$ le morphisme de projection. D'apr\`es \cite[Chapter 2, Proposition 7.11 (a)]{hartshorne} on a un isomorphisme \[p_* \mathcal{O}_{\P}(1) \simeq \mc{O}(-d) \oplus \mc{O}(-d) \oplus \mathcal{O}\] et on d\'efinit $u \in H^0(\P, \mathcal{O}_{\P}(1))$ comme l'image de la section $1\in A=H^0(S,\mc{O})$.
La formule de projection nous donne un isomorphisme 
\[ p_*(p^*\mc{O}(d) \otimes \mathcal{O}_{\P}(1)) \simeq \mathcal{O} \oplus \mathcal{O} \oplus \mc{O}(d),\]
et on d\'efinit $s, v \in H^0(\P, p^*\mc{O}(d) \otimes \mathcal{O}_{\P}(1))$ comme les images de $1\in H^0(S,\mc{O})$, vu comme premier et deuxi\`eme facteur, respectivement. Les sections $s^2$, $v^2$ et $u^2F$ appartiennent alors \`a $H^0(\P,p^*\mc{O}(2d)\otimes \mc{O}_\P(2))$, et on pose
\[Q_{A,F}\coloneqq \{s^2+v^2=u^2F\}\subset \P.\]

 \begin{lemma}\label{refBW}
Soient $q_1,q_2\in H^0(\P^2_\R, \mc{O}(2))$ deux formes quadratiques positives telles que les coniques lisses $q_1=0$ et $q_2=0$ dans $\P^2_\R$ soient transverses. Soit $Y\coloneqq Q_{\R, q_1q_2}$ et soit $\pi \colon  Y \to \P^2_{\R}$ le morphisme de projection. Alors $Y$ satisfait les propri\'et\'es (a)-(d) du th\'eor\`eme \ref{generateur}.
 \end{lemma}

 \begin{proof}
 Le lieu singulier de $Y$ est donn\'e par l'\'equation $q_{1}=q_{2}=s=v=0$ et donc consiste en un nombre fini de points ferm\'es non r\'eels. Ceci montre (a) et (c). 
 
 L'application $Y(\R) \to \P^2(\R)$ est surjective et ses fibres
 sont connexes. Ainsi $Y(\R)$ est connexe. Ceci donne (b).

Montrons (d). La $\C$-vari\'et\'e $Y_\C$ \'etant birationnelle \`a $s'v'= q_1q_2$, elle est rationnelle. Donc $Z_\C$ est aussi rationnelle et en particulier $\Br(Z_\C)=0$. 

Soient $h\in H^0(\P^2_\R,\mc{O}(1))$ et $p_i\coloneqq q_i/h^2\in \R(\P^2)$ pour $i=1,2$. Soit \[\alpha\coloneqq (-1,p_1)\in H^2(\R(\P^2), \Z/2).\] Par un calcul de r\'esidus sur $q_1=0$ et $q_2=0$, on v\'erifie que cette classe ne provient pas d'une classe dans $\Br(\R)$, et qu'elle n'est pas \'egale \`a la classe $(-1, p_1p_2)$.
 Puisque le noyau de l'application $H^2(\R(\P^2), \Z/2)\to H^2(\R(Z), \Z/2)$ est engendr\'e par la classe $(-1, p_1p_2)$, on d\'eduit que l'image $\beta$ de $\alpha$ dans $H^2(\R(Z), \Z/2)$ n'est pas nulle. Puisque $Z$ est lisse, pour montrer que l'on a $$\beta \in \Br (Z)[2]\subset H^2(\R(Z), \Z/2),$$ il suffit de montrer que $\beta$ est non ramifi\'ee \cite[Proposition 4.2.3(a)]{colliot1995birational}. 
 
 Soit $w$ une valuation discr\`ete de $\R(Y)=\R(Z)$, soit $\kappa(w)$ son corps r\'esiduel et soit $c_w$ le centre de $w$ sur $\P^2_\R$. On voit que si $c_w$ 
n'est pas  sur la conique $q_1=0$, alors le r\'esidu de $\beta$ est nul. Si $c_w$ est le point g\'en\'erique de $q_1=0$ et $w(p_{1})$ est pair, alors le r\'esidu de $\beta$ est nul. 
Si $c_w$ est le point g\'en\'erique de $q_1=0$ et $w(p_{1})$ est impair, alors on v\'erifie sur l'\'equation que $-1$ est un carr\'e dans $\kappa(w)$. Si $c_w$ est un point ferm\'e, alors $\kappa(w)=\C$, car la conique n'a pas de point r\'eel, et  $-1$ est encore un carr\'e dans $\kappa(w)$.
\end{proof}

Soit $f=f(x,y,z)$ un polynôme homog\`ene r\'eel de degr\'e $4$ d\'efinissant une courbe lisse dans $\P^2_\R$ et soit $F\coloneqq q_1q_2+tf\in H^0(\P^2_{\R[t]},\mc{O}(4))$. Soit $X$ le fibr\'e en coniques sur $\P^2_{\puiseux}$ donn\'e par l'\'equation
\[s^2+v^2=F(x,y,z).\]
Plus pr\'ecis\'ement, on pose $X\coloneqq Q_{\puiseux,F}$. Consid\'erons $\mc{X}\coloneqq Q_{\R[t],F}$ avec sa projection naturelle $\mc{X}\to \P^2_{\R[t]}$. On a $\mc{X}|_{t=0}=Y$ et $\mc{X}\times_{\R[t]}\puiseux=X$. Un calcul facile montre que la $\puiseux$-vari\'et\'e $X$ est lisse. Soit $C=\C\{\!\{t\}\!\}$. Alors un ouvert affine de $X_C$ est donn\'e par l'\'equation $s'v'=F(x,y,1)$, o\`u $s'=s+iv$ et $v'=s-iv$, ce qui entraîne la $C$-rationalit\'e de $X_C$. 

Le th\'eor\`eme \ref{generateur} et le lemme \ref{refBW} donnent ici: 

 \begin{thm}\label{prop-fibres-coniques}
La $R$-vari\'et\'e projective, lisse et g\'eom\'etriquement rationnelle
$X\coloneqq \mc{X}\times_{\R[t]}R$, o\`u $R$ est le corps des s\'eries de Puiseux r\'eelles en $t$, est un solide fibr\'e en coniques sur $\P^2_R$ tel que $X(R)$ soit semi-alg\'ebriquement connexe mais la vari\'et\'e $X$ n'est pas $CH_0$-triviale, et en particulier n'est pas stablement rationnelle.
\end{thm}

\section{La m\'ethode de sp\'ecialisation revisit\'ee, II}\label{methode-2}

\begin{lemma}\label{lemme-phi-isomorphisme-variante} 
 On se place dans la situation \ref{situation-phi} et on suppose de plus que les homomorphismes $p_*\colon CH_0(T_{E})\to CH_0(S_{E})$ et $p_*\colon CH_0(T_{E'})\to CH_0(S_{E'})$ sont des isomorphismes. Alors l'application $\Phi$ est un isomorphisme.
 \end{lemma}

\begin{proof}
Consid\'erons le diagramme commutatif
\begin{displaymath}
\begin{adjustbox}{max width=\textwidth}
\begin{tikzcd}[row sep=2em, column sep=3em]
CH_1(p^{-1}(U)_{E'},1) \arrow[d,"N_{E'/E}"] \arrow[ddd, bend left=66,"\wr"', "p_*"] \arrow[r] & 
CH_{0}(T_{E'}) \arrow[r] \arrow[d,"N_{E'/E}"] \arrow[ddd, bend left=66,"\wr"', "p_*"] & 
CH_{0}(Z_{E'}) \arrow[r] \arrow[d,"N_{E'/E}"] \arrow[ddd, bend left=66,"p_*"] & 
CH_{0}(p^{-1}(U)_{E'}) \arrow[r] \arrow[d,"N_{E'/E}"] \arrow[ddd, bend left=66,"\wr"', "p_*"] & 
0 \\
CH_1(p^{-1}(U)_E,1) \arrow[r] \arrow[d,"\wr"', "p_*"] & 
CH_{0}(T_{E}) \arrow[r] \arrow[d,"\wr"', "p_*"] & 
CH_{0}(Z_{E}) \arrow[r] \arrow[d,"p_*"] & 
CH_{0}(p^{-1}(U)_{E}) \arrow[r] \arrow[d,"\wr"', "p_*"] & 
0 \\
CH_1(U_{E},1) \arrow[r] & 
CH_{0}(S_{E}) \arrow[r] & 
CH_{0}(Y_{E}) \arrow[r] & 
CH_{0}(U_{E}) \arrow[r] & 
0 \\
CH_1(U_{E'},1) \arrow[r] \arrow[u,swap,"N_{E'/E}"] & 
CH_{0}(S_{E'}) \arrow[r] \arrow[u,swap,"N_{E'/E}"] & 
CH_{0}(Y_{E'}) \arrow[r] \arrow[u,swap,"N_{E'/E}"] & 
CH_{0}(U_{E'}) \arrow[r] \arrow[u,swap,"N_{E'/E}"] & 
0
\end{tikzcd}
\end{adjustbox}
\end{displaymath}
o\`u les complexes horizontaux sont exacts et o\`u les applications $p_*\colon CH_0(T_E)\to CH_0(S_E)$ et $p_*\colon CH_0(T_{E'})\to CH_0(S_{E'})$ sont des isomorphismes par hypoth\`ese. Ici $CH_1(-,1)$ d\'esigne le groupe de Chow sup\'erieur au sens de Bloch \cite{bloch1986algebraic}; voir \cite[Proposition 1.3]{bloch1986algebraic} pour la covariance par morphismes propres et \cite[Theorem 3.1]{bloch1986algebraic} pour la suite de localisation. (On aurait pu utiliser le formalisme de Rost \cite{rost1996chow}: voir \cite[(3.4) et p. 365]{rost1996chow} pour la covariance par morphismes propres, et voir \cite[p. 356]{rost1996chow} pour la suite de localisation.) On en d\'eduit  que dans le carr\'e commutatif
\[
\begin{tikzcd}
    CH_0(Z_{E'}) \arrow[d,"\wr"', "p_*"] \arrow[r,"N_{E'/E}"] & CH_0(Z_E) \arrow[d,"\wr"', "p_*"] \\
    CH_0(Y_{E'}) \arrow[r,"N_{E'/E}"]  & CH_0(Y_E)
\end{tikzcd}
\]
les fl\`eches verticales sont des isomorphismes, ce qui entraîne que l'application
\[\Phi\colon  CH_{0}(Z_{E})/ N_{E'/E}(CH_{0}(Z_{E'})) \to CH_{0}(Y_{E})/N_{E'/E}(CH_{0}(Y_{E'}))\]
est un isomorphisme.
\end{proof}

 \begin{thm}\label{generateur-variante}
Soit $Y$ une $\R$-vari\'et\'e projective et g\'eom\'etriquement int\`egre. Soit $S\subset Y$ le lieu singulier de $Y$, avec structure r\'eduite. Soit $U\coloneqq Y\setminus S$ l'ouvert compl\'ementaire de $S$, soit $p\colon Z\to Y$ une r\'esolution des singularit\'es telle que la restriction $p^{-1}(U)\to U$ soit un isomorphisme, et soit $T\coloneqq p^{-1}(S)$. Soit $i\geq 0$ un entier. Supposons:

\begin{enumerate}
    \item[(a)] les applications 
    \[p_*\colon CH_0(T_{\R(Y)})\to CH_0(S_{\R(Y)}),\qquad p_*\colon CH_0(T_{\C(Y)})\to CH_0(S_{\C(Y)})\] sont des isomorphismes;
    \item[(b)] l'espace topologique $Y(\R)$ est connexe;
    \item[(c)] l'application
\[H^i(\R,\Z/2) \to \Ker [H^i_{\on{nr}}(\R(Y)/\R,\Z/2) \to H^i_{\on{nr}}(\C(Y)/\C,\Z/2)]\] n'est pas surjective.
\end{enumerate}

Soit $\A^1_\R=\Spec(\R[t])$, soit $U\subset \A^1_\R$ un voisinage Zariski de $0$, soit $f\colon \mc{X}\to U$ un morphisme projectif de $\R$-vari\'et\'es, \`a fibre g\'en\'erique lisse, tel que $\mc{X}|_{t=0}\cong Y$. Alors, si $\eta\colon \Spec(\puiseux)\to\A^1_\R$ est le morphisme correspondant \`a l'inclusion $\R[t]\subset \puiseux$, la $\puiseux$-vari\'et\'e projective lisse $X\coloneqq \mc{X}\times_{U,\eta}\Spec(\puiseux)$ satisfait les propri\'et\'es suivantes.

\begin{enumerate}
    \item[(i)] L'espace semi-alg\'ebrique $X(\puiseux)$ est semi-alg\'ebriquement connexe.
    \item[(ii)] La $\puiseux$-vari\'et\'e $X$ n'est pas $CH_0$-triviale, et en particulier
n'est pas stablement rationnelle.
\end{enumerate} 
\end{thm}

 \begin{proof}
     On suit la preuve du th\'eor\`eme \ref{generateur}, en remplaçant le lemme \ref{lemme-phi-isomorphisme} par le lemme \ref{lemme-phi-isomorphisme-variante}. 
 \end{proof}

\begin{rmk}\label{fibre-par-fibre}
    On rappelle une condition \enquote{fibre par fibre} qui entraîne l'hypoth\`ese  du lemme \ref{lemme-phi-isomorphisme-variante} (et donc, pour $k=\R$, l'hypoth\`ese (a)  du th\'eor\`eme \ref{generateur-variante}) et qui sera utilis\'ee dans la suite. Soient $V$ et $W$ deux vari\'et\'es sur un corps $k$ et soit $f\colon V\to W$ un $k$-morphisme propre. On suppose que, pour tout point $P$ du sch\'ema $W$, de corps r\'esiduel $k(P)$, la fibre $V_P$ est une $k(P)$-vari\'et\'e $CH_0$-triviale. D'apr\`es \cite[Proposition 1.8]{colliot2016hypersurfaces}, pour tout corps $F$ contenant $k$, l'application $f_*\colon CH_0(V_F)\to CH_0(W_F)$ est un isomorphisme.
\end{rmk}

\section{Exemples en dimension sup\'erieure}\label{dimsup}

\subsection{Le mod\`ele singulier} Soit $\P^1_\R=\on{Proj}(\R[u_0,u_1])$, vu comme recollement de $\A^1_\R$ avec coordonn\'ee $u\coloneqq u_0/u_1$ et de $\A^1_\R$ avec coordonn\'ee $v\coloneqq u_0/u_1$ au moyen du changement de variable $u=1/v$ sur $\P^1_\R\setminus\{0,\infty\}$. Soit $Y\subset \P^1_\R\times \P^7_\R$ la vari\'et\'e d\'efinie par l'\'equation
\[X_0^2+X_1^2+X_2^2 + (u_0^2+u_1^2)X_3^2 - u_0u_1(X_4^2+X_5^2+X_6^2+X_7^2) =0,\]
o\`u $X_0,\dots,X_7$ sont des cordonn\'ees sur $\P^7_\R$. La projection $Y\to \P^1_\R$ est une fibration en quadriques de dimension relative $6$. Un mod\`ele affine de $Y$ est donn\'e par l'\'equation
\begin{equation} \label{eq:affine}
x^2+y^2+z^2 + (1+u^2) - u(x'^2+y'^2+z'^2+t'^2) = 0.
\end{equation}
 
\begin{lemma}\label{deuxieme-modele}
L'espace topologique $Y(\R)$ est connexe et l'application 
\[H^3(\R,\Z/2)\to \on{Ker}[H^3_{\on{nr}}(\R(Y),\Z/2)\to H^3_{\on{nr}}(\C(Y),\Z/2)]\] n'est pas surjective.
\end{lemma}

 \begin{proof}
L'image de $Y(\R) \to \P^1(\R)$ est l'intervalle $u\geq 0$, et pour tout $u\geq 0$ la fibre $Y_u(\R)$ est connexe. Donc $Y(\R)$ est connexe.

Soit $q$ la forme quadratique
$\langle 1,1,1,1+u^2, -u,-u,-u,-u\rangle$ sur le corps $\R(u)$.
La forme $q$ est de rang $8$ et  n'est pas semblable \`a une voisine
d'une forme  de Pfister, puisque son d\'eterminant
n'est pas un carr\'e.

D'apr\`es Arason \cite[Satz 5.6]{arason1975cohomologische} (voir aussi \cite[Corollaire 9.6.2]{kahn2008formes}), pour une telle forme quadratique $q$,
l'application $H^3(\R(u),\Z/2)  \to H^3(\R(u)(q),\Z/2)$
 est injective.
 
Soit $\R\subset A\subset \R(Y)$ un anneau de valuation discr\`ete
de corps des fractions $\R(Y)$  et de corps r\'esiduel $\kappa$.
Soit $v$ la valuation associ\'ee \`a $A$. 

Supposons
$(-1,-1)_{\kappa} \neq 0$, c'est-\`a-dire que $-1$ n'est pas une somme de
deux carr\'es dans $\kappa$. Alors une somme non nulle de $4$ carr\'es dans 
le compl\'et\'e de $\R(Y)$ en $v$ a sa valuation paire.
Donc  $v(x'^2+y'^2+z'^2+t'^2)$ est paire.
Supposons $v(u)$  impaire. Dans le compl\'et\'e de $A$,
si $v(u)>0$, alors $1+u^2$ est un carr\'e. Si $v(u)<0$, alors 
$1+u^2$ est aussi un carr\'e dans le compl\'et\'e.  Donc 
$x^2+y^2+z^2 + 1+u^2$ est une somme de $4$ carr\'es dans le compl\'et\'e,
donc de valuation paire. Mais ceci contredit l'\'equation.
Ainsi, pour toute valuation $v$ telle que le corps r\'esiduel $\kappa$
satisfasse $(-1,-1)_{\kappa} \neq 0$, la valuation $v(u)$ est paire.

On conclut que la classe $(-1,-1,u) \in H^3(\R(Y),\Z/2)$ a tous ses r\'esidus nuls,
elle est donc non ramifi\'ee. D'autre part, si $(-1,-1,u) \in H^3(\R(Y),\Z/2)$ provenait de $H^3(\R,\Z/2\Z)$ alors, l'application $H^3(\R(u),\Z/2)\to H^3(\R(Y),\Z/2)$ \'etant injective, la classe $(-1,-1,u) \in H^3(\R(u),\Z/2)$ proviendrait de $H^3(\R,\Z/2\Z)$. Ceci  est impossible car la classe $(-1,-1,u) \in H^3(\R(u),\Z/2)$  est ramifi\'ee.
\end{proof}

\subsection{Vari\'et\'es lisses fibr\'ees en quadriques}

Soit $\mc{E}=\oplus_{i=0}^7\mc{E}_i$ un fibr\'e vectoriel sur $\P^1_\R$, o\`u $\mc{E}_i\coloneqq \mc{O}(1)$ pour $0\leq i\leq 2$ et $\mc{E}_i=\mc{O}$ pour $3\leq i\leq 7$. Soit encore $\mc{L}\coloneqq \mc{O}(2)$. On a $\mc{E}_i^\vee\otimes \mc{E}_i^\vee\otimes \mc{L}\cong \mc{O}$ si $0\leq i\leq 2$ et $\mc{E}_i^\vee\otimes \mc{E}_i^\vee\otimes \mc{L}\cong \mc{O}(2)$ si $3\leq i\leq 7$. On dispose donc des sections globales 
\begin{align*}
    1&\in \Gamma(\P^1_\R,\mc{O})=\Gamma(\P^1_\R, \mc{E}_i^\vee\otimes \mc{E}_i^\vee\otimes \mc{L})\qquad &(0\leq i\leq 2), \\
    u_0^2+u_1^2&\in \Gamma(\P^1_\R,\mc{O}(2))=\Gamma(\P^1_\R, \mc{E}_i^\vee\otimes \mc{E}_i^\vee\otimes \mc{L})\qquad &(i=3),\\
    -u_0u_1& \in \Gamma(\P^1_\R,\mc{O}(2))=\Gamma(\P^1_\R, \mc{E}_i^\vee\otimes \mc{E}_i^\vee\otimes \mc{L})\qquad &(3\leq i\leq 7).
\end{align*}

Consid\'erons la fibration en quadriques $Q \subset \P(\mc{E})$ d\'efinie par l'annulation de la section globale 
\[
q \coloneqq Z_0^2 + Z_1^2 + Z_2^2 + (u_0^2+u_1^2)Z_3^2 - u_0u_1(Z_4^2 + Z_5^2 + Z_6^2 + Z_7^2)\in \Gamma(\P^1_\R,\mathrm{Sym}^2(\mc{E}^\vee) \otimes \mc{L}).
\]

Soit $U_1=\{u_1\neq 0\}\subset \P^1_\R$. Une base de $\mc{E}|_{U_1}$ en tant que $\mc{O}_{U_1}$-module libre est donn\'ee par $f_0/u_1, f_1/u_1, f_2/u_1, f_3, f_4, f_5, f_6, f_7$, o\`u $f_i$ est une section globale de $(\mc{E}_i)|_{U_1}$. Notons $X_i$ les coordonn\'ees de fibre correspondantes. Alors $(Z_i)|_{U_1}=u_1X_i$ si $0\leq i\leq 2$ et $(Z_i)|_{U_1}=X_i$ si $3\leq i\leq 7$.
En substituant dans l'\'equation $q=0$, on obtient l'\'equation locale de $Q$ dans $U_1 \times \P^7_\R$:
\[
 u_1^2(X_0^2+X_1^2+X_2^2) + (u_0^2+u_1^2)X_3^2 - u_0u_1(X_4^2+X_5^2+X_6^2+X_7^2) = 0.
\]
Ceci \'equivaut \`a
\begin{equation} \label{eq:local_S}
 X_0^2+X_1^2+X_2^2 + (1+u^2)X_3^2 - u(X_4^2+X_5^2+X_6^2+X_7^2) = 0,
\end{equation}
o\`u $u=u_0/u_1$. En particulier, la vari\'et\'e affine (\ref{eq:affine}) est isomorphe \`a l'ouvert de Zariski de $Q$ o\`u $u_1\neq 0$ et $Z_3\neq 0$, en identifiant
\[(x,y,z,x',y',z',t')=(X_0/X_3,X_1/X_3,X_2/X_3,X_4/X_3,X_5/X_3,X_6/X_3,X_7/X_3).\]

Soit $U_0=\{u_0\neq 0\}\subset \P^1_\R$. Un argument analogue nous donne l'\'equation locale de $Q$ dans $U_0 \times \P^7_\R$:
\[
u_0^2(Y_0^2+Y_1^2+Y_2^2) + (u_0^2+u_1^2)Y_3^2 - u_0u_1(Y_4^2+Y_5^2+Y_6^2+Y_7^2) = 0,
\]
o\`u $(Z_i)|_{U_0}=u_0Y_i$ si $0\leq i\leq 2$ et $(Z_i)|_{U_0}=Y_i$ si $3\leq i\leq 7$.
Ceci \'equivaut \`a
\begin{equation} \label{eq:local_T}
Y_0^2+Y_1^2+Y_2^2 + (1+v^2)Y_3^2 - v(Y_4^2+Y_5^2+Y_6^2+Y_7^2) = 0,
\end{equation}
o\`u $v=u_1/u_0$.

\begin{lemma}\label{fibration-quadriques-a-c}
    La $\R$-vari\'et\'e $Y=Q\subset\P(\mc{E})$ satisfait les hypoth\`eses (a), (b) et (c) (avec $i=3$) du th\'eor\`eme \ref{generateur-variante}.
\end{lemma}

\begin{proof}
La $\R$-vari\'et\'e $Q$ est birationnelle \`a la $\R$-vari\'et\'e du lemme \ref{deuxieme-modele}. La propri\'et\'e (c) en r\'esulte.

Le lieu singulier $S_1$ de $Q|_{U_1}$ est donn\'e par les \'equations
\begin{align*}
    X_0=X_1=X_2=(1+u^2)X_3=uX_4=uX_5=uX_6=uX_7&=0, \\
    2uX_3^2-(X_4^2+X_5^2+X_6^2+X_7)^2&=0,
\end{align*}
et donc co\"incide avec
\[S_1=\{u=X_0=X_1=X_2=X_3=X_4^2+X_5^2+X_6^2+X_7^2=0\}.\]
Le lieu singulier $S_0$ de $Q|_{U_0}$ est donn\'e par
\[S_0=\{v=Y_0=Y_1=Y_2=Y_3=Y_4^2+Y_5^2+Y_6^2+Y_7^2=0\}.\]
Le lieu singulier $S$ de $Q$ est donc l'union disjointe de $S_0$ et $S_1$. En particulier $Q(\R)$ est contenu dans le lieu non singulier de $Q$. En particulier, on obtient (b): la $\R$-vari\'et\'e $Q$ est birationnelle \`a la vari\'et\'e du lemme \ref{deuxieme-modele} et donc le lemme \ref{connexe-invariant} entraîne que $Q(\R)$ est connexe.

Les \'equations (\ref{eq:local_S}) et (\ref{eq:local_T}) montrent que le fibr\'e en quadriques $Q\to \P^1_\R$ est admissible au sens de \cite[D\'efinition 3.1]{colliot1993groupe}, c'est-\`a-dire, sa fibre g\'en\'erique est lisse et pour tout point ferm\'e $P\in \P^1_\R$, la restriction de $Q\to \P^1_\R$ \`a $\Spec(O_P)$ est le lieu des z\'eros d'une \'equation homog\`ene $\sum_{i=1}^8 a_iT_i^2$ o\`u $0\leq v_P(a_i)\leq 1$ pour tout $i$ et $v_P(a_i)=0$ pour $1\leq i\leq 4$. D'apr\`es \cite[Th\'eor\`eme 3.3(a)]{colliot1993groupe}, une r\'esolution des singularit\'es $\tilde{Q}\to Q$ est obtenue en \'eclatant la vari\'et\'e $Q$ le long de son lieu singulier. De plus, pour tout point singulier $P\in Q$ la fibre du morphisme $\tilde{Q}\to Q$ en $P$ est donn\'ee par une hypersurface quadrique lisse avec un $k(P)$-point rationnel (cf. \cite[bas de page 488]{colliot1993groupe}). Compte tenu de la remarque \ref{fibre-par-fibre}, ceci entraîne (a).
\end{proof}

Soient $q_3(u),\dots,q_7(u)\in \R[u]$ des polynômes de degr\'e $2$, sans racines doubles et tels que $\on{pgcd}(q_i,q_j)=1$ pour tout $3\leq i<j\leq 7$. Posons
\[p_i(u_0,u_1)\coloneqq u_1^2q_i(u_0/u_1)\in \Gamma(\P^1_\R,\mc{O}(2))=\Gamma(\P^1_\R,\mc{E}_i^\vee\otimes \mc{E}_i^\vee\otimes \mc{L})\qquad (3\leq i\leq 7).\]
Soit
\[q'\colon Z_0^2+Z_1^2+Z_2^2+p_3Z_3^2+p_4Z_4^2+p_5Z_5^2+p_6Z_6^2+p_7Z_7^2\in \Gamma(\P^1_\R, \on{Sym}^2(E^\vee)\otimes L),\]
et soit $Q'\coloneqq \{q'=0\}\subset \P(\mc{E})$. L'ouvert $(Q')_{U_1}\subset U_1\times \P^7_\R$ est donn\'e par l'\'equation
\[X_0^2+X_1^2+X_2^2+q_3(u)X_3^2+q_4(u)X_4^2+q_5(u)X_5^2+q_6(u)X_6^2+q_7(u)X_7^2=0,\]
o\`u $(Z_i)|_{U_1}=u_1X_i$ si $0\leq i\leq 2$ et $(Z_i)|_{U_1}=X_i$ si $3\leq i\leq 7$.
Le lieu singulier de $(Q')_{U_1}$ est donc donn\'e par
\begin{align*}
    X_0=X_1=X_2=q_3(u)X_3=q_4(u)X_4=q_5(u)X_5=q_6(u)X_6=q_7(u)X_7&=0,\\
    (\partial q_3)(u)X_3+(\partial q_4)(u)X_4+(\partial q_5)(u)X_5+(\partial q_6)(u)X_6+(\partial q_7)(u)X_7&=0,
\end{align*}
o\`u $\partial q_i$ est la d\'eriv\'ee de $q_i$ en $u$. Soit $P\in (Q')_{U_1}$ un point singulier. D'apr\`es les hypoth\`eses faites sur les $q_i$, il existe au plus un $j$ tel que $q_j(u(P))=0$, ce qui entraîne $X_i=0$ en $P$ pour tout $i\neq j$. Comme $q_j$ n'a pas de racine double, on a $(\partial q_j)(u(P))\neq 0$, et la derni\`ere \'equation donne alors $X_j=0$ en $P$, contradiction. Comme les $q_i$ sont de degr\'e $2$, la fibre de $Q'$ en $(1:0)$ est lisse. On conclut que la $\R$-vari\'et\'e $Q'$ est lisse. 

On consid\`ere la famille $\mc{X}\to \on{Spec}(\R[t])$ d\'efinie par l'\'equation $q+tq'=0$ dans $\P(\mc{E})\times_\R\R[t]$. La fibre du morphisme $\mc{X}\to \on{Spec}(\R[t])$ en $t=0$ est $Q$ et la fibre g\'en\'erique est lisse. On pose $X\coloneqq \mc{X}\times_{\R[t]}\puiseux$. Le th\'eor\`eme \ref{generateur-variante} et le lemme \ref{fibration-quadriques-a-c} nous donnent:

\begin{thm}\label{fibres-quadriques-6}
    La $\puiseux$-vari\'et\'e $X$ est une vari\'et\'e lisse, projective, de dimension $7$,  fibr\'ee en quadriques sur $\P^1_{\puiseux}$,  non $CH_0$-triviale, et telle que $X(\puiseux)$ soit semi-alg\'ebriquement connexe.
\end{thm}

\begin{rmk}\label{autrescontrex}  Au th\'eor\`eme \ref{prop-fibre-quadriques} nous avons construit une vari\'et\'e 
lisse, projective, de dimension $3$,  fibr\'ee en quadriques sur $\P^1_{\puiseux}$ avec les mêmes propri\'et\'es.
   On part l\`a d'une fibration 
en quadriques $Y \to \P^1_\R$ dont la fibre g\'en\'erique est d\'efinie par une forme quadratique $q$ de rang $r=4$ sur $F=\R(\P^1)$
dont le d\'eterminant n'est pas un carr\'e. On suppose en outre
que $Y$ poss\`ede un $\R$-point lisse.
On utilise 
l'injectivit\'e de $H^2(F,\Z/2) \to H^2(F(q), \Z/2)$,
qui r\'esulte de l'hypoth\`ese sur le d\'eterminant,
et une classe non constante dans $H^2_{\on{nr}}(F(q)/\R, \Z/2)$. 

Pour toute forme quadratique $q$ non d\'eg\'en\'er\'ee de rang $r\geq 5$
l'application $$H^2(F,\Z/2) \to H^2_{\on{nr}}(F(q)/F, \Z/2)$$
est un isomorphisme. 
Pour $r\geq 5$, Ceci implique que l'application
$$H^2(\R,\Z/2) \to H^2_{\on{nr}}(F(q)/\R, \Z/2)$$
est un isomorphisme. 
On ne peut donc utiliser le groupe
$H^2(-, \Z/2)$ pour faire une construction analogue
avec $X$ de dimension plus grande que 3.
 
La  construction du th\'eor\`eme \ref{fibres-quadriques-6}
part d'une fibration
en quadriques $Y \to \P^1_\R$ dont la fibre g\'en\'erique est d\'efinie par une forme quadratique de rang $r= 8$ sur $F=\R(\P^1)$
dont le d\'eterminant n'est pas un carr\'e, donc qui n'est pas semblable \`a une 3-forme de Pfister. 
On utilise le fait que pour une telle forme~$q$ l'application
$H^3(F,\Z/2) \to H^3(F(q), \Z/2)$ est injective.

Plus g\'en\'eralement, d'apr\`es Arason \cite[Satz 5.6]{arason1975cohomologische},
 cette application est injective si la forme $q$
est de rang $r\geq 5$ et n'est pas une 3-voisine de Pfister, i.e.
n'est pas semblable \`a une sous-forme d'une 3-forme
de Pfister $\langle \!\langle a,b,c \rangle \!\rangle$.
Par ailleurs, on sait \cite[Th\'eor\`eme 10.2.4]{kahn2008formes} que pour  $q$ avec $r\geq 3$, l'application
$$H^3(F, \Q/\Z(2)) \to H^3_{\on{nr}}(F(q)/F, \Q/\Z(2))$$
est surjective sauf si $r=6$ et $q$
 est une forme d'Albert anisotrope.
 On d\'eduit de ces r\'esultats et d'un calcul de r\'esidus
 que l'application $H^3(\R, \Z/2) \to H^3_{\on{nr}}(F(q)/\R, \Z/2)$
 est surjective si $r\geq 9$.

On voit ainsi qu'outre les exemples que nous avons d\'evelopp\'es avec $r=4$  en utilisant $H^2$ et  $r=8$ en utilisant $H^3$,  on pourrait essayer d'utiliser $H^3(-, \Z/2)$ pour construire des  exemples avec  $r=5,6,7$ et $q$ non voisine d'une 3-forme de Pfister. 
\end{rmk}

\subsection{Intersections de deux quadriques lisses} 

Soit $\A^9_\R$ d\'efini par les coordonn\'ees $(x, y, z, u, v, x', y', z', t')$ et soit $U$ la vari\'et\'e r\'eelle affine d\'efinie par les \'equations:
\begin{align}
x^2+y^2+z^2+1+u^2 - uv &= 0, \label{eq:U1} \\
x'^2+y'^2+z'^2+t'^2 - v &= 0. \label{eq:U2}
\end{align}
Soient $[X_0: X_1: X_2: X_3: X_4: X_5: X_6: X_7: X_8: X_9]$ des coordonn\'ees homog\`enes sur $\P^9_\R$. On identifie $\A^9_\R$ \`a $U_0=\{X_0\neq 0\}\subset \P^9_\R$ en posant
\[(x,y,z,u,v,x',y',z',t')=(X_1/X_0,X_2/X_0,\dots, X_9/X_0).\]
D\'efinissons $Y$ par:
\begin{align} 
X_0^2+X_1^2+X_2^2+X_3^2+X_4^2 - X_4X_5 &= 0, \label{eq:F1} \\
X_6^2+X_7^2+X_8^2+X_9^2 -X_0X_5 &= 0. \label{eq:F2}
\end{align}
Donc $U=Y\cap U_0$. 

\begin{lemma}\label{int-2-quadriques-a-c}
    La $\R$-vari\'et\'e $Y$ satisfait les hypoth\`eses (a), (b) et (c) (avec $i=3$) du th\'eor\`eme \ref{generateur-variante}. 
\end{lemma}

\begin{proof}
La $\R$-vari\'et\'e $Y$ est birationnelle \`a la $\R$-vari\'et\'e du lemme \ref{deuxieme-modele}. La propri\'et\'e (c) en r\'esulte.

La matrice jacobienne associ\'ee au syst\`eme d'\'equations (\ref{eq:F1})-(\ref{eq:F2}) est :
$$ \begin{pmatrix} 2X_0 & 2X_1 & 2X_2 & 2X_3 & 2X_4-X_5 & -X_4 & 0 & 0 & 0 & 0 \\ X_5 & 0 & 0 & 0 & 0 & X_0 & -2X_6 & -2X_7 & -2X_8 & -2X_9 \end{pmatrix}. $$
Soit $P$ un point singulier de $Y$. Si $X_5\neq 0$ en $P$, alors 
\[X_1=X_2=X_3=X_4=2X_4-X_5=2X_0^2+X_4X_5=0,\] ce qui contredit (\ref{eq:F1}). Donc $X_5=0$ en $P$. Si $X_0\neq 0$ en $P$, alors \[0=X_4X_5=-2X_0^2\neq 0,\] contradiction. Donc $X_0=X_5=0$ en $P$. La forme de la matrice jacobienne entraîne alors que l'on a soit $X_1=X_2=X_3=X_4=0$ en $P$ (avec $X_6,X_7,X_8,X_9$ arbitraires) soit $X_6=X_7=X_8=X_9=0$ (avec $X_1,X_2,X_3,X_4$ arbitraires). On d\'eduit que le lieu singulier $S\subset Y$, avec la structure r\'eduite, est l'union disjointe de
\begin{align*}
    S_1\coloneqq \{ X_0=X_1=X_2=X_3=X_4=X_5=0, \quad X_6^2+X_7^2+X_8^2+X_9^2=0\},\\
    S_2\coloneqq\{ X_0=X_5=X_6=X_7=X_8=X_9=0, \quad X_1^2+X_2^2+X_3^2+X_4^2=0\}. 
\end{align*}
    En particulier, $Y(\R)$ est contenu dans le lieu lisse de $Y$ et connexe, et donc par le lemme \ref{connexe-invariant} donne (b). 

Soit $p\colon Z\to Y$ l'\'eclatement de $Y$ en $S$: on veut montrer que $p$ est une r\'esolution des singularit\'es satisfaisant la propri\'et\'e (a). On montrera le r\'esultat plus pr\'ecis suivant: la $\R$-vari\'et\'e $Z$ est lisse et pour tout $P\in S$ la fibre de $p$ en $P$ est une quadrique projective lisse avec un $k(P)$-point (d'apr\`es la remarque \ref{fibre-par-fibre}, ceci entraîne (a)). Pour tout $0\leq i\leq 9$, soit $U_i\coloneqq \{X_i\neq 0\}\subset \P^9_\R$. Pour tout point $P\in S_1$ il existe $i$ avec $6\leq i\leq 9$ tel que $P\in U_i$. Par sym\'etrie, on peut supposer $i=9$. L'ouvert $Y\cap U_9\subset \A^9_\R$ est donn\'e par les \'equations
\begin{align*}
    x_0^2+x_1^2+x_2^2+x_3^2+x_4^2-x_4x_5 &= 0,\\
    x_6^2+x_7^2+x_8^2+1-x_0x_5&= 0,
\end{align*}
o\`u $x_i\coloneqq X_i/X_9$ sont les cordonn\'ees de $U_9=\A^9_\R$. 

Comme on a $P\in S_1$, on a $x_0=0$ en $P$ et donc la deuxi\`eme \'equation entraîne l'existence de $j$ avec $6\leq j\leq 8$ tel que $x_j\neq 0$ en $P$. Par sym\'etrie, on peut supposer  $j=8$. Soit $\A^9_\R$ un autre espace affine, avec cordonn\'ees $y_0,\dots,y_8$. On a un morphisme
\[f\colon \A^9_\R\to \A^9_\R,\qquad (x_0,\dots,x_8)\mapsto (x_0,\dots,x_7,x_8^2).\]
Comme $x_8\neq 0$ en $P$, la restriction de $f$ \`a $Y\cap U_9\subset \A^9_\R$ est \'etale au voisinage de $P$. De plus, si $W\subset \A^9_\R$ est donn\'e par
\begin{align*}
    y_0^2+y_1^2+y_2^2+y_3^2+y_4^2-y_4y_5&=0,\\
    y_6^2+y_7^2+y_8+1-y_0y_5&=0,
\end{align*}
alors $f^{-1}(W)=Y\cap U_9$. On note que $W$ est isomorphe \`a la sous-vari\'et\'e de $\A^8_\R$, avec cordonn\'ees $z_j$ pour $0\leq j\leq 7$, donn\'ee par l'\'equation
\[z_0^2+z_1^2+z_2^2+z_3^2+z_4^2-z_4z_5=0,\]
ce qui est isomorphe \`a $\tilde{C}\times \A^2_\R$, o\`u $\tilde{C}$ est le cône affine sur une quadrique projective lisse $C$ avec un $\R$-point. En particulier, l'\'eclat\'e de la vari\'et\'e $\tilde{C}\times \A^2_\R$ en son lieu singulier $\{0\}\times \A^2_\R$ est une $\R$-vari\'et\'e lisse, et le diviseur exceptionnel est donn\'e par la deuxi\`eme projection $C\times \A^2_\R\to \A^2_\R$. On d\'eduit que tout point de $p^{-1}(S_1)$ est un point lisse de $Z$ et que pour tout $P\in S_1$ la fibre de $p$ en $P$ est une quadrique projective lisse avec un $k(P)$-point. On peut traiter $S_2$ de fa\c{c}on similaire. D'apr\`es la remarque
\ref{fibre-par-fibre}, ceci \'etablit
la propri\'et\'e (a).
\end{proof}

Soient $q_1$ et $q_2$ deux formes quadratiques sur $\R$ en les $9$ variables $X_0,\dots,X_9$, telles que le syst\`eme $q_1=q_2=0$ d\'efinisse une intersection compl\`ete lisse de deux quadriques dans $\P^9_{\R}$. 

Soit $\mc{X}\subset \P^9_{\R[t]}$ d\'efini par le syst\`eme
\begin{align*} 
X_0^2+X_1^2+X_2^2+X_3^2+X_4^2 - X_4X_5 +tq_1=0, \\
X_6^2+X_7^2+X_8^2+X_9^2 -X_0X_5 + tq_2=0.
\end{align*}
Alors $\mc{X}|_{t=0}=Y$ et la fibre g\'en\'erique sur $\R(t)$ est une intersection compl\`ete lisse de deux quadriques dans $\P^9_{\R(t)}$.
On pose $X\coloneqq \mc{X}\times_{\R[t]}\puiseux$.
Le th\'eor\`eme \ref{generateur-variante} et le lemme \ref{int-2-quadriques-a-c} donnent ici :

\begin{thm}\label{intersection-quadriques-p9}
L'intersection lisse de deux quadriques $X$ dans dans $\P^9_{\puiseux}$ n'est pas $CH_0$-triviale, et l'espace $X(\puiseux)$ est semi-alg\'ebriquement connexe. 
\end{thm}

\begin{rmk}
Au th\'eor\`eme  \ref{prop-intersection-quadriques} nous avons construit des intersections compl\`etes lisses $X$ de deux
quadriques dans $\P^5_{\puiseux}$ pour lesquelles $X(\puiseux)$ est semi-al\-g\'e\-bri\-que\-ment connexe et $X$ n'est pas $CH_0$-triviale. Soit $R$ un corps r\'eel clos. 
On peut se demander s'il existe de tels exemples dans $\P^n_R$ pour d'autres valeurs $n\geq 6$, $n\neq 9$. 
On pourrait essayer d'en construire pour
$n=6,7,8$ en partant de la remarque \ref{autrescontrex}.
Il convient cependant de se rappeler que d'apr\`es \cite{hassett2022rationality}
de tels exemples n'existent pas dans $\P^6_R$.
    \end{rmk}

\section{Calcul d'invariants cohomologiques non ramifi\'es pour les exemples}\label{calcul}

\begin{prop}\label{br-surj}
Soit  $k$ un corps de caract\'eristique z\'ero.  
Soit $X$ une $k$-vari\'et\'e projective lisse
g\'eom\'e\-tri\-quement int\`egre. Supposons que $X$ est $k$-birationnelle \`a une $k$-vari\'et\'e
de l'un des types suivants :

\begin{itemize}
    \item[(a)] intersection compl\`ete  lisse de deux quadriques dans $\P^n_{k}$ avec $n \geq 5$;
    \item[(b)] hypersurface cubique lisse dans $\P^n_{k}$ avec $n \geq 4$;
    \item[(c)] mod\`ele projectif et lisse d'une $k$-vari\'et\'e g\'eom\'etriquement int\`egre $Y$ munie d'une
fibration en quadriques $Y \to \P^1_k$ \`a fibre g\'en\'erique lisse et dont toutes les fibres ont une composante g\'eom\'etriquement int\`egre de multiplicit\'e $1$.
\end{itemize} 

Alors  $\Br(k) \to \Br(X)$ est surjectif.
\end{prop}

\begin{proof}
    
 Dans les cas (a) et (b),  cette fl\`eche est un isomorphisme \cite[Theorem 8.3.2]{colliot2021brauer}.
Dans le cas (c), on commence par observer que la fl\`eche
\[\Br(k(\P^1)) \to \Br_{\on{nr}}(k(Y)/k(\P^1))= \Br(Y_{\eta})\] est surjective car la fibre g\'en\'erique $Y_{\eta}/k(\P^1)$
est une quadrique lisse. Une classe $\alpha$  dans $\Br(X) \subset \Br(k(X))= \Br(k(Y))$
est donc image d'une classe $\beta \in \Br(k(\P^1))$. Pour tout point ferm\'e $P\in \P^1_k$, la fibre $Y_P$ admet une composante irr\'eductible $Z_P\subset Y_P$ de multiplicit\'e $1$. La comparaison du r\'esidu de $\beta$ en $P$ et du r\'esidu de $\alpha$ au point g\'en\'erique de $Z_P$ (voir \cite[Corollary 11.1.6]{colliot2021brauer})  donne que le r\'esidu de $\beta$ en $P$ est trivial. Comme ceci pour tout point ferm\'e $P$ de $\P^1_k$, on conclut que $\alpha$ est image d'un \'el\'ement $\gamma \in  \Br(\P^1_{k})=\Br(k)$.
\end{proof} 

\subsection{Un th\'eor\`eme de Witt}

\begin{thm}\label{wittgen}
    Soit $R$ un corps r\'eel clos, et soit $K=R(\Gamma)$ le corps des fonctions d'une $\R$-courbe
$\Gamma$ g\'eom\'etriquement int\`egre. Soient $q_{1}$ et $q_{2}$ deux formes quadratiques   non d\'eg\'en\'er\'ees
sur le corps $K$. Supposons $1 \leq {\rm rang}(q_{1}) <  {\rm rang}(q_{2})$ et $3 \leq {\rm rang}(q_{2})$.
Si apr\`es \'evaluation en  presque tout point  $c \in \Gamma(R)$, la forme quadratique
$q_{1,c}$  est isomorphe \`a une sous-forme de $q_{2,c}$, alors la $K$-forme $q_{1}$
est isomorphe \`a une sous-forme de $q_{2}$. En particulier, si une forme quadratique $q$
non d\'eg\'en\'er\'ee sur $K$ de rang au moins 3 se sp\'ecialise en une forme $q_{c}$
isotrope sur $R$ pour presque tout $c$, alors cette forme sur $K$ a un z\'ero non trivial.
\end{thm}

\begin{proof} 
Ceci se d\'eduit en combinant deux r\'esultats de Witt : 
 \cite[Satz 22]{witt1937theorie} sur l'isotropie des formes quadratiques sur $R(\Gamma)$,
  ce qui correspond au cas du plan hyperbolique $q_{1}(x,y)=x^2-y^2$, et \cite[Satz 4]{witt1937theorie} sur la simplification des formes  quadratiques.
  Comme l'espace des repr\'esentations
 de $q_{1}$ par $q_{2}$ est un espace homog\`ene sous le groupe lin\'eaire
 connexe $SO(q_{2})$, cet \'enonc\'e est  aussi  un cas particulier de \cite[Corollaire 6.2]{scheiderer1996hasse}.
 \end{proof}
 
 \begin{rmk}
 Cet \'enonc\'e est en d\'efaut si le rang de $q_{1}$ est \'egal au rang de $q_{2}$.
 \end{rmk}

 \subsection{Fibrations en quadriques et intersections de deux quadriques}

On calcule $H^3_{\on{nr}}(-,\Z/2)$ pour certaines fibrations en quadriques. En utilisant le th\'eor\`eme suivant, on en d\'eduit des r\'esultats analogues pour certaines intersections lisses de deux quadriques avec un point rationnel. 

\begin{thm}\label{theorem3.2}
    Soit $k$ un corps parfait de caract\'eristique diff\'erente de $2$, soit $X\subset \P^n_k$ une intersection compl\`ete lisse de deux quadriques telle que $X(k)$ soit non vide. \`A tout $k$-point $P$ de $X$, on associe une fibration en quadriques $X'\to\P^1_k$, o\`u $X'$ est lisse, telle que toute fibre g\'eom\'etrique ait au plus un point singulier (et donc soit int\`egre si $n\geq 5$) et telle que $X'$ soit $k$-birationnelle \`a $X$.
\end{thm}

\begin{proof}
    Voir \cite[Theorem 3.2, (3.4) p. 61, Remark 1.13.1]{colliot1987intersectionsI}.
\end{proof}

Nous utiliserons aussi le th\'eor\`eme suivant.
\begin{thm}\label{KRS}
Soit $k$ un corps de caract\'eristique diff\'erente de $2$.
Soit $q$ une forme quadratique non d\'eg\'en\'er\'ee sur $k$, de rang $r\geq 3$. Soit $k(q)$ le corps des fonctions de la quadrique projective d\'efinie par $q=0$.
Soit $i \geq 0$ un entier. Dans chacun des cas suivants, l'application
$$H^{i}(k,\Q/Z(i-1)) \to H^{i}_{\on{nr}}(k(q), \Q/Z(i-1))$$
est surjective:

(i) $i=2$ et $r\geq 3$

(ii) $i=3$ et $r\neq 6$

(iii) $i=3$, $r=6$  et $q$ n'est pas une forme d'Albert

(iv) $i=4$ et $r\leq 5.$

\end{thm}
\begin{proof} Le cas $i=2$ est classique. Les \'enonc\'es pour $i=3$ sont dus \`a Kahn, Rost et Sujatha \cite[Th\'eor\`eme 5, Corollaire 10 (2)]{kahn1998unramified}. L'\'enonc\'e (iv) est dû \`a Kahn et Sujatha \cite[Theorem 3]{kahnsuja1998unramified}. 
\end{proof}

\begin{prop}\label{H3nrRpincquad}
 Soit $R$ un corps r\'eel clos.
 Soit $X$ une $R$-vari\'et\'e projective et lisse munie d'un
 morphisme  $X \to \P^1_{R}$ dont la fibre g\'en\'erique est une quadrique de dimension  $d$.  Supposons $X(R)$ semi-alg\'ebriquement connexe.

(a) Si $d=2$,  l'application $H^3(R,\Z/2) \to H^3_{\on{nr}}(R(X)/R,\Z/2)$ est un isomorphisme.


(b) Si $d=3$, l'application
$H^4(R,\Z/2) \to H^4_{\on{nr}}(R(X)/R,\Z/2)$
est un isomorphisme.
\end{prop}
\begin{proof}
D\'emontrons d'abord l'\'enonc\'e (a).
Notons $C=R(\sqrt{-1})$. La fibration $X_{C}  \to \P^1_{C}$ admet une section, sa fibre g\'en\'erique est une quadrique, donc la $C$-vari\'et\'e $X_{C}$ est rationnelle. Pour tout corps $k$ de caract\'eristique diff\'erente de $2$, d'apr\`es le th\'eor\`eme de Merkurjev  \cite{merkurjev1981norm}, l'application $H^2(k,\mu_{2^n}^{\otimes 2})\to H^2(k,\mu_{2^{n-1}}^{\otimes 2})$ est surjective pour tout $n\geq 1$, et donc l'application $H^3(k,\Z/2) \to H^3(k,\Q/\Z(2))$ est injective. On en d\'eduit l'injectivit\'e des fl\`eches horizontales dans le diagramme commutatif suivant
\[
\begin{tikzcd}
    H^3(R(\P^1),\Z/2) \arrow[r,"\sim"] \arrow[d] & H^3(R(\P^1),\Q/\Z(2)) \arrow[d,->>] \\
    H^3_{\on{nr}}(R(X)/R(\P^1),\Z/2) \arrow[r,hook] & H^3_{\on{nr}}(R(X)/R(\P^1),\Q/\Z(2)).
\end{tikzcd}
\]
La surjectivit\'e de la fl\`eche du haut suit de la nullit\'e de $H^3(C(\P^1),\Q/\Z(2))$ et d'un argument de restriction-corestriction. La surjectivit\'e de l'application de droite 
est donn\'ee par le th\'eor\`eme \ref{KRS}(ii). On en d\'eduit la surjectivit\'e de la fl\`eche de gauche. Combinant ces faits, on voit que toute classe $$\alpha \in 
 H^3_{\on{nr}}(R(X)/R,\Z/2) \subset H^3_{\on{nr}}(R(X)/R(\P^1),\Z/2)$$ est l'image d'une classe $\beta \in
 H^3(R(\P^1), \Z/2)$.
 Soit $\alpha \in H^3_{\on{nr}}(R(X),\Z/2)$. On fixe un point  $m$ de $X(R)$. Sous l'hypoth\`ese que $X(R)$ est semi-alg\'ebriquement  connexe,
 l'image de $X(R)$ dans $\P^1(R)$ est un intervalle de $\P^1(R)$.
 Si cet intervalle est tout $\P^1(R)$, le th\'eor\`eme \ref{wittgen} assure que la fibration $X \to \P^1_{R}$ admet une section
 et donc que l'espace total est $R$-rationnel. On
 supposera que ce n'est pas le cas.
 On  peut supposer l'intervalle  donn\'e par $0 \leq u \leq \infty$,
 o\`u $\A^1_{R}=\Spec(R[u]) \subset \P^1_{R}$.
 
Soit $G\coloneqq \on{Gal}(C/R)$. Comme le corps $C(u)$ est de dimension cohomologique $1$, la suite exacte longue de cohomologie associ\'ee \`a la suite de $G$-modules \[0\to \Z/2\to (\Z/2)[G]\to \Z/2\to 0\] montre la surjectivit\'e de l'application $H^i(R(u),\Z/2)\to H^{i+1}(R(u),\Z/2)$ donn\'ee par $a\mapsto (-1)\cup a$ pour tout $i\geq 1$. Donc toute classe dans $H^3(R(u),\Z/2)$ s'\'ecrit
 comme un cup-produit $(-1,-1,P(u))$ avec $P(u)\in R[u]$. Comme $(-1,r^2+s^2)=0$ dans $\on{Br}(R(u))$ pour tout $r,s\in R(u)$,  on peut de plus supposer
 $$P(u) = \prod_{i=1}^n (u+a_{i})$$
 avec tous les $a_{i} \in R$ distincts. La classe $\alpha$ prend  une valeur constante  dans $\Z/2= H^3(R, \Z/2)$ sur chaque composante semi-alg\'ebrique connexe de $X(R)$, donc sur $X(R)$. Quitte \`a modifier $\alpha$ par l'image de $(-1,-1,-1)$, on peut supposer que $\alpha$ s'annule sur $X(R)$.   Ainsi la fonction $P(u) $ est positive sur $u>0$. Ceci implique $a_{i} \geq 0$ pour tout $i$.

On peut supposer  la fibre g\'en\'erique $X_{\eta}$ de $X\to \P^1_R$ donn\'ee par l'annulation
d'une forme quadratique $q\coloneqq \langle 1,a,b,c\rangle$, avec $a,b,c \in R(u)$.

Soit  $f\in R(u)$.
Si la forme quadratique non d\'eg\'en\'er\'ee  $q$ est une sous-forme de la forme quadratique
$\langle \!\langle -1,-1, f\rangle \!\rangle $,  alors cette forme de Pfister admet un z\'ero sur le corps des fonctions
$R(u)(q)$ de la quadrique d\'efinie par $q$, donc est totalement hyperbolique sur
ce corps. Ceci implique que  la classe de cohomologie $(-1,-1,f) \in H^3(R(u),\Z/2)$
s'annule dans $H^3(R(u)(q),\Z/2)$.
 
 Pour terminer la d\'emonstration, il suffit donc de montrer :
 pour tout  $a \geq 0$, la forme quadratique $q$, de rang $4$,
 est une sous-forme de la forme  $\langle \!\langle -1,-1, u+a \rangle \!\rangle$, de rang 8.
D'apr\`es le th\'eor\`eme \ref{wittgen}, il suffit de voir
que ceci vaut pour presque toute sp\'ecialisation de $u$
en un point $c\in R$.
 
 Soit $a\geq 0$.
 Si $c+a  <0$, alors $c<0$ et donc $c$ n'est pas dans l'image de 
 $X(R) \to \P^1(R)$. La forme $q_{c}$ est donc anisotrope.
 Comme $q$ repr\'esente 1 sur $R(u)$, il en est de m\^{e}me de $q_{c}$ sur $R$,
 et $q_{c}$ est isomorphe \`a $\langle 1,1,1,1\rangle $.
 La forme $\langle \!\langle-1,-1, c+a \rangle \!\rangle$ contient la forme 
 $\langle \!\langle -1,-1 \rangle \!\rangle=\langle 1,1,1,1 \rangle$.
 
 Si $c+a>0$,  alors la forme  $\langle \!\langle -1,-1, c+a\rangle \!\rangle$  sur $R$ est isotrope donc totalement
 hyperbolique. Elle repr\'esente donc toute forme quadratique sur $R$
 de rang au plus \'egal \`a 4.

Voici les modifications  \`a apporter \`a cette d\'emonstration pour obtenir l'\'enonc\'e (b). On consid\`ere le diagramme commutatif
 \[
\begin{tikzcd}
    H^4(R(\P^1),\Z/2) \arrow[r,"\sim"] \arrow[d] & H^4(R(\P^1),\Q/\Z(3)) \arrow[d,->>] \\
    H^4_{\on{nr}}(R(X)/R(\P^1),\Z/2) \arrow[r,hook] & H^4_{\on{nr}}(R(X)/R(\P^1),\Q/\Z(3)).
\end{tikzcd}
\]
Pour tout corps $k$ de caract\'eristique diff\'erente de $2$,  l'application $H^3(k,\mu_{2^n}^{\otimes 3})\to H^3(k,\mu_{2^{n-1}}^{\otimes 3})$ est surjective pour tout  $n\geq 1$ par un th\'eor\`eme de Merkurjev, Suslin et Rost,
 et donc l'application $H^4(k,\Z/2) \to H^4(k,\Q/\Z(3))$ est injective. Ceci donne l'injectivit\'e des fl\`eches horizontales.
 La surjectivit\'e de la fl\`eche verticale de droite est donn\'ee par le th\'eor\`eme \ref{KRS} (iv). Le groupe $H^4(R(\P^1),\Q/\Z(3))$ est de 2-torsion car $C(\P^1)$ est de dimension cohomologique 1.  Soit $q$ une forme quadratique de rang 5 sur $R(u)$ repr\'esentant 1 d\'efinissant la fibre g\'en\'erique de $X \to \P^1_R$. En utilisant  le th\'eor\`eme \ref{wittgen},
 et en proc\'edant comme pour la d\'emonstration de (a), on voit, par sp\'ecialisation de $u$ en presque tout point $c\in R$,
 que pour tout $a\geq 0$, la forme quadratique
 $q$, de rang 5, et donc de rang au plus 8,  est une sous-forme de la forme $\langle \!\langle -1,-1,-1, u+a \rangle \!\rangle$, de rang 16.
 \end{proof}
\begin{rmk}
La surjectivit\'e de la fl\`eche 
$H^{4}(k,\Q/\Z(3)) \to H^{4}_{\on{nr}}(k(q), \Q/\Z(3))$
est  aussi \'etablie par Kahn, Rost et Sujatha pour certains types de formes quadratiques de rang $r>5$. Pour celles qui sont de rang au plus 8, la d\'emonstration de la proposition \ref{H3nrRpincquad}(b) s'\'etend.
\end{rmk}

\begin{rmk}
En combinant \cite[Theorem 1.1]{benoist2024rationality} et \cite[Th\'eor\`eme 4.7]{colliot2024certaines}, on obtient sur $R$ le corps des s\'eries
de Puiseux r\'eelles  un
exemple de fibration $X \to \P^1_{R}$  \`a fibres g\'eom\'etriques des surfaces quadriques int\`egres,
et d'un  corps $F$ contenant $R$, tels que $X(R)$ soit semi-alg\'ebriquement connexe, et que
 l'application $H^3(F,\Z/2) \to H^3_{\on{nr}}(F(X)/F,\Z/2)$
ne soit pas surjective.
\end{rmk}

\begin{rmk}
La proposition  \ref{H3nrRpincquad} (a) g\'en\'eralise \cite[Th\'eor\`eme 4.5]{colliot2024certaines}.
Comme observ\'e dans \cite[Remarque 4.6]{colliot2024certaines},
Benoist et Wittenberg ont \'etabli que la fl\`eche
$H^3(R,\Z/2) \to H^3_{\on{nr}}(R(X)/R,\Z/2)$
est un isomorphisme pour de bien plus larges classes
de solides $X/R$. Plus pr\'ecis\'ement,
  sur tout corps
r\'eel clos $R$, la combinaison de
\cite[Proposition 5.2]{benoist2020integral1} et \cite[Theorem 8.1(i)]{benoist2020integral2} \'etablit que pour
tout solide projectif et lisse g\'eom\'etriquement unir\'egl\'e et satisfaisant
$H^2(X,O_{X})=0$, si $X(R)$ est semi-alg\'ebriquement connexe, alors
$H^3(R,\Q/\Z(2)) \to H^3_{\on{nr}}(X,\Q/\Z(2))$ est un isomorphisme.
\end{rmk}

\begin{cor}\label{H3Rnrdeuxquad}
Soit $R$ un corps r\'eel clos et soit $X \subset \P^5_{R}$ une intersection compl\`ete lisse de deux quadriques. Supposons $X(R) \neq \emptyset$. Supposons $X(R)$ semi-alg\'ebrique\-ment connexe.
Alors l'application $ H^3(R, \Z/2) \to H^3_{\on{nr}}(R(X), \Z/2)$
est un isomorphisme.
\end{cor} 

\begin{proof}
     On combine le th\'eor\`eme \ref{theorem3.2} et la proposition \ref{H3nrRpincquad}.
\end{proof}

\begin{prop}\label{fibP1Rrelplus4}
Soit $F$ un corps de caract\'eristique diff\'erente de $2$.
Soit $X\to \P^1_{F}$ une fibration en quadriques de dimension relative $n\geq 4$. 
On suppose que $X/F$ est projective et lisse
et  que  les fibres g\'eom\'etriques singuli\`eres  ont un seul point singulier
(pinceau de Lefschetz).
On suppose en outre que la fibre g\'en\'erique, qui est d\'efinie par
une forme quadratique $q$ sur $F(\P^1)$  de rang $n+2\geq 6 $,
n'est pas semblable \`a une forme d'Albert anisotrope.
Supposons $X(F)\neq \emptyset$.
Alors  l'application $H^3(F,\Z/2) \to H^3_{\on{nr}}(F(X)/F, \Z/2)$ est
un isomorphisme.
\end{prop}

\begin{proof}
Rappelons qu'une forme d'Albert est de rang $6$. Si la forme $q$ est isotrope, alors la fibre g\'en\'erique est une quadrique lisse avec un point rationnel, donc est rationnelle sur le corps $F(\P^1)$,
et $X$ est rationnelle sur $F$. Le r\'esultat est alors clair. Supposons d\'esormais $q$ anisotrope. 
Soit $\alpha \in H^3_{\on{nr}}(F(X)/F, \Z/2)$. Comme $q$
n'est pas semblable \`a une forme d'Albert, d'apr\`es 
\cite[Theorem 5]{kahn1998unramified},
$\alpha$ est l'image d'une classe $\beta \in H^3(F(\P^1),\Z/2)$.
Les r\'esidus de cette classe en tout point  ferm\'e $m \in \P^1_{k}$
sont dans le noyau de $H^2(F(m), \Z/2) \to H^2(F(X_{m}), \Z/2)$.
Comme chaque $X_{m}$ est d\'efini par une forme quadratique de rang au moins 5,
ce noyau est trivial. Ainsi $\beta$ est dans $H^3(F,\Z/2)$.
\end{proof}

\begin{cor}\label{2quadsurFngeq7} 
Soit $F$ un corps de caract\'eristique diff\'erente de $2$, soit $n\geq 7$ et soit $X \subset \P^n_F$ une intersection compl\`ete lisse de deux quadriques. Supposons $X(F) \neq \emptyset$. 
Soit $Y \to \P^1_F$ la fibration en quadriques associ\'ee \`a un $F$-point de $X$ comme dans le th\'eor\`eme \ref{theorem3.2}.
Si $n=7$, supposons que la fibre g\'en\'erique de cette fibration n'est pas d\'efinie par une forme d'Albert anisotrope sur $F(\P^1)$. Alors l'application $H^3(F,\Z/2) \to H^3_{\on{nr}}(F(X)/F, \Z/2)$ est un isomorphisme.
\end{cor}
\begin{proof}
On combine le th\'eor\`eme \ref{theorem3.2} et la proposition \ref{fibP1Rrelplus4}. 
\end{proof}


\begin{prop}\label{fibP1Rrel3}
Soit $R$ un corps r\'eel clos et soit $X\to \P^1_{R}$ une fibration en quadriques de dimension relative $n= 3$. 
On suppose que la vari\'et\'e  $X/R$, de dimension $4$, est projective et lisse
et  que  chaque fibre g\'eom\'etrique a au plus un point singulier. Supposons que $X(R)$ est semi-alg\'ebriquement connexe.
Alors pour tout corps $F$ contenant $R$,
 l'application $H^3(F,\Z/2) \to H^3_{\on{nr}}(F(X)/F, \Z/2)$ est
un isomorphisme.
\end{prop}
 \begin{proof}

 Notons $T\subset \P^1(R)$ l'ensemble des points $P$ tels que la fibre $X_P$ soit d\'efinie par une forme de rang $4$ sur $R$ anisotrope. Une \'equation locale de $X\to \P^1_R$ au voisinage d'un point $P\in T$ est de la forme $\sum_{i=0}^3 a_ix_i^2+bx_4^2=0$, o\`u $a_i\in O_{\P^1_R,P}^\times$ et $b$ est une uniformisante de $O_{\P^1_R,P}$ et $a_i(P)>0$ pour tout $i$.
Pour tout $Q$ proche de $P$, l'ensemble $X_Q(R)$ est vide ou non selon que $b(P)>0$ ou $b(P)<0$. Par contre, soit $P'\in \P^1(R)\setminus T$. On v\'erifie sur l'\'equation locale au voisinage de $P'$ que pour $Q'$ proche de $P'$ la vacuit\'e de $X_{Q'}(R)$ ne d\'epend pas de $Q'$. On d\'eduit que la cardinalit\'e $|T|$ de $T$ est paire. Si on avait $|T|\geq 4$, alors $X(R)$ aurait au moins $2$ composantes semi-alg\'e\-bri\-quement connexes. Donc $|T|$ est \'egal \`a $0$ o\`u $2$. Si $T$ est vide, alors la projection $X(R) \to \P^1(R)$ est surjective
et le th\'eor\`eme \ref{wittgen} assure que $X \to \P^1_{R}$ a une section
et donc que $X$ est rationnelle sur $R$.
Supposons que $|T|=2$.
On peut donc supposer $T=\{0,\infty\}$.

 Soit $\alpha \in H^3_{\on{nr}}(F(X)/F, \Z/2)$. D'apr\`es \cite[Theorem 5]{kahn1998unramified} (cf. preuve de la proposition \ref{H3nrRpincquad}),
$\alpha$ est l'image d'une classe $\beta \in H^3(F(\P^1),\Z/2)$.
Les r\'esidus de cette classe en tout point  ferm\'e $m \in \P^1_{F}$
sont dans le noyau de $H^2(F(m), \Z/2) \to H^2(F(X_{m}), \Z/2)$.
Ce noyau est trivial si $X_{m}/F$ est lisse ou si $X_m$ contient un point
$F(m)$-rationnel lisse. Soit  $m$  un point ferm\'e de $\P^1_{F}$ 
dont la fibre $X_{m}/F(m)$ n'est pas lisse et est d\'efinie par
une forme quadratique de rang $4$ anisotrope.
Alors l'image de $m$ par la projection  $\P^1_{F}\to \P^1_{R}$ est
un point $P\in T$, donc $m=0$ ou $m=\infty$.

 Le r\'esidu de $\alpha$ en tout point ferm\'e $m$ de $\P^1_{F}\setminus \{0,\infty\}$ est trivial, car ce r\'esidu est 
dans le noyau de $H^2(F(m),\Z/2) \to H^2(F(X_{m}),\Z/2)$.

En $0$ et $\infty$, ce r\'esidu est $0$ ou la classe de $(-1,-1)_{F}$.
De plus, par r\'eciprocit\'e \cite[Proposition 2.2]{rost1996chow}, la somme de ces deux r\'esidus est  nulle. Donc
\`a addition pr\`es d'un \'el\'ement de $H^3(F,\Z/2)$ on a
  $\beta=0$  ou $\beta=(-1,-1,u)_{F}$, o\`u $u$ est une fonction rationnelle sur $\P^1_{R}$
de diviseur $(0)-(\infty)$. Pour conclure, il suffit de montrer que l'image de $(-1,-1,u)_{R}$ dans  $H^3(R(X), \Z/2)$ est nulle.
 
 On peut supposer que la fibre g\'en\'erique  de $X\to \P^1_R$ est d\'efinie par une forme quadratique $q$ sur $R(\P^1)$  de rang $5$ qui repr\'esente $1$.
  Il suffit de montrer que la forme $q$ est une sous-forme de $\langle \!\langle -1,-1,u\rangle \!\rangle$.
  D'apr\`es le th\'eor\`eme \ref{wittgen}, ceci  vaut si pour presque tout $c\in R$,
  la forme $q_{c}$ est une sous-forme de $\langle \!\langle -1,-1,c\rangle \!\rangle$.
  Pour presque tout $c>0$ la forme  $q_{c}$ de rang $5$ est isotrope et la forme de Pfister $\langle \!\langle-1,-1,c\rangle \!\rangle$ est isotrope donc
  totalement hyperbolique. Ainsi $q_{c}$ est une sous-forme de $\langle \!\langle-1,-1,c\rangle \!\rangle$.
  Pour presque tout $c<0$, la forme  $q_{c}$ de rang $5$ est anisotrope et repr\'esente 1, donc
  est isomorphe \`a $ \langle 1,1,1,1,1\rangle $. La forme   $\langle \!\langle-1,-1,c\rangle \!\rangle$ est isomorphe
  \`a  $\langle \!\langle-1,-1,-1\rangle \!\rangle$. Ainsi $q_{c}$ est une sous-forme de $\langle \!\langle-1,-1,c\rangle \!\rangle$.
     \end{proof}

     \begin{prop}\label{fibquaRconnexenqeq3}
     Soit $R$ un corps r\'eel clos et soit $X\to \P^1_{R}$ une fibration en quadriques de dimension relative $ n \geq  3$. 
On suppose que la vari\'et\'e  $X/R$ est projective et lisse
et  que  chaque fibre g\'eom\'etrique a au plus un point singulier.
Supposons que $X(R)$ est semi-alg\'ebrique\-ment connexe.
Pour toute extension $F/R$, l'application $H^3(F,\Z/2) \to H^3_{\on{nr}}(F(X)/F, \Z/2)$ est
un isomorphisme.
\end{prop}
\begin{proof}
Pour $n=3$, c'est la proposition \ref{fibP1Rrel3}. Pour $n\geq 4$,
cela r\'esulte imm\'ediatement de la proposition \ref{fibP1Rrelplus4}, sauf si
$n=4$ et la forme quadratique $q$ d\'efinissant la fibre g\'en\'erique
est une forme d'Albert anisotrope. Mais sur le corps $R(\P^1)$,
indice et exposant des \'el\'ements du groupe de Brauer co\"{\i}ncident
(tout \'el\'ement est annul\'e par passage \`a $R(\sqrt{-1})$). Ceci implique
qu'un produit tensoriel de deux alg\`ebres de quaternions n'est pas un corps
gauche, et ce dernier fait implique (Albert)  que toute forme d'Albert sur $R(\P^1)$ est isotrope.
\end{proof}

\begin{cor}
 Soit $X \subset \P^n_R$ une intersection compl\`ete lisse de deux quadriques sur un corps r\'eel clos $R$. Supposons $n \geq 6$ et $X(R)$ semi-alg\'ebriquement connexe.
Pour tout corps $F$ contenant $R$, l'application 
$H^3(F,\Z/2) \to H^3_{\on{nr}}(F(X)/F, \Z/2)$ est un isomorphisme.
\end{cor}
\begin{proof} 

On combine le th\'eor\`eme \ref{theorem3.2} et la proposition \ref{fibquaRconnexenqeq3}.
\end{proof}

\begin{rmk}
  Dans le cas $n=6$,  Hassett, Koll\'ar et Tschinkel \cite{hassett2022rationality}
  ont montr\'e bien mieux :  toute telle $X \subset \P^6_R$ est rationnelle sur $R$.
\end{rmk}

Soit $R$ un corps r\'eel clos et $X$ une $R$-vari\'et\'e projective et lisse g\'eom\'etriquement connexe de dimension $n$.
Soit $s\geq 0$ le nombre de ses composantes semi-alg\'ebriques connexes. 
On a \'etabli dans \cite{colliot1990real} que l'on a
   $ H^{i}_{\on{nr}}(R(X)/R,\Z/2) \simeq (\Z/2)^s $
   pour $i\geq n+1$. En particulier, si $X(R)$ est semi-alg\'ebriquement connexe et $i\geq n+1$, alors l'application $H^i(R,\Z/2)\to H^i_{\on{nr}}(R(X)/R,\Z/2)$ est un isomorphisme.  
   
   Rassemblons ici les r\'esultats
   obtenus sur l'invariant $ H^{i}_{\on{nr}}(R(X)/R,\Z/2)$
   pour les fibrations en quadriques $X$ sur $\P^1_R$
   de dimension relative au plus $3$.

\begin{cor} 
Soit $R$ un corps r\'eel clos.
 Soit $X$ une $R$-vari\'et\'e projective et lisse munie d'un
 morphisme  $X \to \P^1_{R}$ dont la fibre g\'en\'erique est une quadrique de dimension  $d\geq 1$.
 Notons $\phi_i : H^{i}(R,\Z/2) \to H^{i}_{\on{nr}}(R(X)/R,\Z/2)$.

(a) Pour tout $d$ et $i\leq 1$, $\phi_i$ est un isomorphisme.

(b) Si $d=1$ et $X(R)$ est semi-alg\'ebriquement connexe,
alors $X$ est rationnelle et $\phi_i$ est un isomorphisme pour tout $i$.

(c) Si $d=2$ et $X(R)$ est semi-alg\'ebriquement connexe,
$\phi_i$ est un isomorphisme si $i\neq 2$. C'est un isomorphisme pour tout $i$ si toutes les fibres g\'eom\'etriques sont des quadriques avec au plus un point singulier.

(d) Si $d=3$ et $X(R)$ est semi-alg\'ebriquement connexe, $\phi_i$ est un isomorphisme si $i\neq 3$. C'est un isomorphisme pour tout $i$ si toutes les fibres g\'eom\'etriques sont des quadriques avec au plus un point singulier.
\end{cor}

\begin{proof}
D'apr\`es le paragraphe pr\'ec\'edent, il suffit de consid\'erer $i\leq d+1$.

(a) Soit $C\coloneqq R(\sqrt{-1})$. On a 
\[H^1(R,\Z/2)=H^1_{\text{\'et}}(X,\Z/2)=H^1_{\on{nr}}(R(X)/R,\Z/2).\]
La premi\`ere \'egalit\'e suit du fait que $H^1_{\text{\'et}}(X_C,\Z/2)=0$, la $C$-vari\'et\'e $X_C$ \'etant rationnelle.

(b) La rationalit\'e de $X$ est un cas particulier du th\'eor\`eme de Comessatti \cite{comessatti1912fondamenti} (cf. \cite[Paragraphe 10.2]{colliot2024certaines}). La bijectivit\'e des $\varphi_i$ s'ensuit.

(c) Pour $i=3$, voir la proposition \ref{H3nrRpincquad} (a). La deuxi\`eme partie de l'\'enonc\'e suit de la proposition \ref{br-surj} (c) et du fait que $X(R)\neq \emptyset$. (Comme on a $d\geq 2$, l'hypoth\`ese additionnelle entraîne que les fibres g\'eom\'etriques sont int\`egres.)

(d) L'application $\varphi_2$ est surjective par la proposition \ref{br-surj} (c) et elle est injective car $X(R)$ est non vide. Le cas $i=4$ suit de la proposition \ref{H3nrRpincquad} (b). La deuxi\`eme partie de l'\'enonc\'e suit de l’int\'egrit\'e des fibres g\'eom\'etriques et de la proposition \ref{fibquaRconnexenqeq3}. 
\end{proof}
 
\subsection{Certaines fibrations en coniques sur \texorpdfstring{$\P^2$}{P2}}
Soit $R$ un corps r\'eel clos. Dans \cite[Proposition 3.4, Corollaire 3.6]{benoist2020clemens}, Benoist et Wittenberg ont \'etabli la non-rationalit\'e de fibrations en coniques de la forme
$$y^2+z^2=g(u,v)$$
lorsque la courbe plane $\Gamma\subset \P^2_R$ d'\'equation affine $g(u,v)=0$ est lisse de genre au moins $2$. Il est naturel de consid\'erer le cas o\`u $\Gamma$ est singuli\`ere et de genre g\'eom\'etrique z\'ero. Ceci motive la proposition suivante.
 
\begin{prop}\label{BeWi}
Soit $R$ un corps r\'eel clos et soit $X \to \P^2_{R}$ une fibration en coniques dont un ouvert affine est
d\'efini par une \'equation
$$y^2+z^2=g(u,v)$$
o\`u $g(u,v)$ est un polyn\^{o}me g\'eom\'etriquement irr\'eductible,
de degr\'e pair, tel que $g\geq 0$ sur $\A^2(R)$.
Supposons que la courbe $\Gamma \subset \P^2_R$ d'\'equation affine $g(u,v)=0$ est \`a singularit\'es quadratiques ordinaires, est de genre  g\'eom\'etrique z\'ero, et ne coupe la droite
\`a l'infini $L$ qu'en des points lisses. Pour tout corps $F$ contenant $R$,
l'application
$$H^3(F,\Q/\Z(2)) \to H^3_{\on{nr}}(F(X)/F,\Q/\Z(2))$$
est un isomorphisme.
\end{prop}

\begin{proof}
Soit $C\coloneqq R(\sqrt{-1})$. La vari\'et\'e $X_{C}$ est rationnelle. Si $-1$ est un carr\'e dans $F$, alors $X_F$ est rationnelle et donc pour tout $i$ l'application $H^i(F,\Q/\Z(2)) \to H^i_{\on{nr}}(F(X)/F,\Q/\Z(2))$ est un isomorphisme. Supposons donc que $\sqrt{-1}$ n'est pas un carr\'e dans $F$. Donc $F'\coloneqq F(\sqrt{-1})$ est une extension de corps de $F$ avec groupe de Galois $G=\on{Gal}(F'/F)\cong \Z/2$.

 Soient $G_F$ le groupe de Galois absolu de $F$ et $M$ un $G_F$-module galoisien discret de torsion premi\`ere \`a $2$. Pour tout entier $i\geq 0$ on a le diagramme commutatif
\[
\begin{tikzcd}
    H^i(F,M) \arrow[r]\arrow[d] & H^i(F',M)^G\arrow[d] \\ 
    H^i_{\on{nr}}(F(X)/F,M) \arrow[r] & H^i_{\on{nr}}(F'(X)/F',M)^G.
\end{tikzcd}
\]
Les fl\`eches horizontales sont des isomorphismes car $2$ est inversible dans $M$ et la fl\`eche verticale de droite est un isomorphisme car $X_{F'}$ est rationnelle sur $F'$. Il s'ensuit que l'application $H^i(F,M)\to H^i_{\on{nr}}(F(X)/F,M)$ est un isomorphisme.
Il nous reste donc a montrer que l'application \[H^3(F,\Q_2/\Z_2(2))\to H^3_{\on{nr}}(F(X)/F,\Q_2/\Z_2(2))\]
est un isomorphisme.

Soit $\alpha \in H^3_{\on{nr}}(F(X)/F,\Q_2/\Z_2(2))$. D'apr\`es Kahn-Rost-Sujatha \cite[Theorem 5]{kahn1998unramified}, il existe $\beta \in H^3(F(\P^2), \Q_2/\Z_2(2))$ d'image $\alpha \in H^3(F(X),\Q_2/\Z_2(2))$. Comme $X(R)$ est non vide, $X(F)$ est non vide. Soit $P\in X(F)$. En modifiant $\alpha$ par un \'el\'ement convenable de $H^3(F,\Q_2/\Z_2(2))$, on peut supposer que la sp\'ecialisation de $\alpha$ en $P$ est nulle. Comme $X_{F'}$ est rationnelle sur $F'$, l'image de $\alpha$ dans $H^3_{\on{nr}}(F'(X),\Q_2/\Z_2(2))$ est alors nulle. L'extension $F'(X)/F'(\P^2)$ \'etant transcendante pure, ceci entraîne que l'image de $\beta$ dans $H^3(F'(\P^2), \Q_2/\Z_2(2))$ est aussi nulle.

Soit $\eta$ le point g\'en\'erique de $\Gamma_F$. On consid\`ere le r\'esidu de $\beta$ en un point $\xi$ 
de codimension $1$ de $\A^2_{F}$ diff\'erent de $\eta$. Ce r\'esidu est dans le sous-groupe de $\Br(F(\xi))$ engendr\'e par $(-1,g)$ o\`u on note $g$ la fonction rationnelle
induite par $g$ en $\xi$. Consid\'erons la classe 
\[(-1,P(u,v),g(u,v)) \in H^3(F(\P^2), \Z/2) \subset H^3(F(\P^2), \Q_2/\Z_2(2))\]
o\`u $P(u,v)$ est un polyn\^{o}me s\'eparable d\'efinissant
les $\xi$ autres que $\eta$ pour lesquels le r\'esidu de $\partial_{\xi}(\beta)$
est non nul. On remplace $\beta$ par $\beta - (-1,P(u,v),g(u,v))$.
Cet \'el\'ement a  encore pour image $\alpha$ dans $H^3(F(X),\Q/\Z(2))$.
On est ramen\'e \`a supposer que $\beta$ a tous ses r\'esidus
sur $\P^2_{F}$ nuls sauf peut-\^{e}tre $\partial_{\eta}(\beta)$
et $\partial_{L}(\beta)$.

Comme $X_{F'}$ est $F'$-rationnelle, l'image de $\alpha$ dans $H^3_{\on{nr}}(F'(X),\Q_2/\Z_2(2))$ provient de $H^3(F',\Q_2/\Z_2(2))$. Le morphisme $X_{F'}\to \P^2_{F'}$ admet une section rationnelle. Par sp\'ecialisation, ceci montre que l'image de $\beta$ dans $H^3(F'(\P^2),\Q_2/\Z_2(2))$ provient de $H^3(F',\Q_2/\Z_2(2))$. 

Soit  $p \colon S \to \P^2_{R}$ l'\'eclatement des points singuliers de $\Gamma$. On consid\`ere les complexes de Bloch-Ogus pour $S_F$ et $S_{F'}$:
\[
\begin{tikzcd}
    H^3(F(S),\Q_2/\Z_2(2))\arrow[r] \arrow[d]  &  H^3(F'(S),\Q_2/\Z_2(2))  \arrow[d] \\
    {\smash{\bigoplus_{\substack{\gamma\in (S_F)^{(1)}}}} H^2(F(\gamma),\Q_2/\Z_2(1))} \arrow[r]  \arrow[d] & 
    {\smash{\bigoplus_{\substack{\gamma'\in (S_{F'})^{(1)}}}} H^2(F'(\gamma'),\Q_2/\Z_2(1))} \arrow[d]  \\
    {\smash{\bigoplus_{\substack{m\in (S_F)^{(2)}}}} H^1(F(m),\Q_2/\Z_2)}     \arrow[r] & 
    {\smash{\bigoplus_{\substack{m'\in (S_{F'})^{(2)}}}} H^1(F'(m'),\Q_2/\Z_2)}.
\end{tikzcd}
\]
Pour tout $m\in (S_F)^{(2)}$ de corps r\'esiduel $F(m)$ contenant $F'$, on a un isomorphisme $F(m)\otimes_F F'\cong F(m_1)\times F(m_2)$, o\`u $m_1,m_2\in (S_{F'})^{(2)}$ sont les deux $F'$-points d'image $m$. Pour $i=1,2$, l'inclusion naturelle $F(m)\subset F(m_i)$ est une \'egalit\'e et donc l'application de restriction $H^1(F(m),\Q_2/\Z_2)\to H^1(F(m_i),\Q_2/\Z_2)$ est un isomorphisme pour $i=1,2$. Comme l'image de $\beta$ dans $H^3(F'(S),\Q_2/\Z_2(2))$ provient de $H^3(F',\Q_2/\Z_2(2))$, on conclut que $\partial_m(\partial_\gamma(\beta))=0$ pour tout $\gamma\in (S_F)^{(1)}$ et pour tout $m\in (S_F)^{(2)}$ dans la clôture de $\gamma$ et corps r\'esiduel $F(m)$ contenant $F'$.

Soit $\Delta \subset S$ la d\'esingu\-laris\'ee de $\Gamma$, soit $m\in \Delta_F$ un point ferm\'e et soit $\cl{m}\in \Gamma$ l'image de $m$. Si $\cl{m}$ n'appartient pas \`a $L_F$ et $\Gamma_F$ est lisse en $\cl{m}$, alors il r\'esulte du complexe de Bloch-Ogus sur $\P^2_F$ que $\partial_m(\partial_\Delta(\beta))=0$. Sinon, comme les points de $\Gamma_F\cap L_F$ et les points singuliers de $\Gamma_F$ sont d\'efinis sur $C$, le corps r\'esiduel $F(m)$ contient $F'$ et donc $\partial_m(\partial_\Delta(\beta))=0$ d'apr\`es le paragraphe pr\'ec\'edent. On conclut que $\partial_{\eta}(\beta) \in H^2(F(\Delta),\Q/\Z(1))$ est non ramifi\'e et donc provient de $H^2(F,\Q/\Z(1))$. Comme l'image de $\beta$ dans $H^3(F'(\Delta),\Q_2/\Z_2(2))$ est nulle, l'image de $\partial_{\eta}(\beta)$ dans $H^2(F'(\Delta),\Q/\Z(1))$ l'est aussi. Comme $\Delta_{F}$ est une conique lisse qui a un point rationnel sur $F'$, on conclut que  $\partial_{\eta}(\beta)$ peut s'\'ecrire $(-1,\rho)$ avec $\rho \in F^*$.

L'\'el\'ement $\beta - (-1,\rho, g) \in H^3(F(\P^2),\Q_2/\Z_2(2))$ a encore pour image $\alpha$
dans le groupe $H^3(F(X),\Q_2/\Z_2(2))$ et a tous ses r\'esidus nuls sur $\P^2_{F}$
sauf peut-\^etre au point g\'en\'erique de $L$. D'apr\`es \cite[Proposition 8.6, Remark 1.11, Remark 2.5]{rost1996chow}, ceci suffit \`a assurer
que $\beta - (-1,\rho, g)$  appartient  \`a $H^3(F,\Z/2)$, et donc
$\alpha$ est dans l'image de $H^3(F,\Q_2/\Z_2(2))$.
\end{proof}

\begin{rmk}
Soit $d$ le degr\'e (pair) de la courbe $\Gamma$. Si $d=2$, la $R$-vari\'et\'e $X$ est une quadrique avec un $R$-point. C'est donc une vari\'et\'e rationnelle. Dans \cite[Theorem 2.2]{benoist2024rationality}, il est montr\'e que la $R$-vari\'et\'e $X$
est rationnelle si $d=4$, et dans \cite[Theorem 4.5]{benoist2024rationality} qu'elle n'est pas rationnelle si $d\geq 12$.
\end{rmk}

\begin{rmk} En degr\'e cohomologique $2$, on a un r\'esultat g\'en\'eral.
    Soient $R$ un corps r\'eel clos, $g(u,v) \in R[u,v]$ un polyn\^{o}me
    g\'eom\'etriquement irr\'eductible et $X$ la $R$-vari\'et\'e affine d'\'equa\-tion $y^2+z^2=f(u,v)$.  Pour toute extension $F/R $ l'application $\Br(F) \to \Br_{\on{nr}}(F(X)/F)$ est surjective. La m\'ethode pour \'etablir ce r\'esultat est analogue \`a celle de la proposition \ref{BeWi} mais plus simple. On utilise le comportement du groupe de Brauer non ramifi\'e  dans une fibration \cite[Chapter 11]{colliot2021brauer}.
    \end{rmk}

\section*{Remerciements}
Les auteurs remercient Olivier Wittenberg pour avoir sugg\'er\'e la deuxi\`eme preuve
du th\'eor\`eme \ref{critere-connexe}. Le deuxi\`eme auteur  a b\'en\'efici\'e du soutien du projet NSF DMS-2201195.

\end{document}